\documentclass[11pt,a4paper,reqno]{article}
\usepackage{amsmath}
\usepackage{amsthm}
\usepackage{enumerate}
\usepackage{amssymb}
\usepackage{hyperref}

\usepackage{verbatim}
\usepackage{url}
\usepackage[latin1]{inputenc}
\usepackage{caption}
\usepackage{graphicx,color}
\def\Ac{{\mathcal A}}
\def\Cc{{\mathcal C}}
\def\Dc{{\mathcal D}}

\def\EE{{\mathbf E}}

\def\be{\begin{equation}}
\def\ee{\end{equation}}
\def\bea{\begin{equation*}}
\def\eea{\end{equation*}}
\def\begs{\begin{split}}
  \def\ends{\end{split}}

\numberwithin{equation}{section}

\newtheorem{thm}{Theorem}%[section]
\newtheorem{lma}[thm]{Lemma}

\newtheorem{prop}[thm]{Proposition}
\newtheorem{claim}{Claim}

\newtheorem{df}[thm]{Definition}

\theoremstyle{remark}
\newtheorem{preremark}[thm]{Remark}
\newtheorem{preex}[thm]{Example}

\title{Strict inequality for the chemical distance exponent in two-dimensional critical percolation}
\author{Michael Damron \thanks{The research of M. D. is supported by NSF grant DMS-0901534 and an NSF CAREER grant.} \\ \small{Georgia Tech}  \and Jack Hanson \thanks{The research of J. H. is supported by NSF grant DMS-1612921.}\\ \small{City College, CUNY} \and Philippe Sosoe\thanks{The research of P. S. is supported by the Center for Mathematical Sciences and Applications at Harvard University.} \\ \small{CMSA, Harvard and Cornell University}}

\begin{document}

\maketitle 
\begin{abstract}
We provide the first nontrivial upper bound for the chemical distance exponent in two-dimensional critical percolation. Specifically, we prove that the expected length of the shortest horizontal crossing path of a box of side length $n$ in critical percolation on $\mathbb{Z}^2$ is bounded by $Cn^{2-\delta}\pi_3(n)$, for some $\delta>0$, where $\pi_3(n)$ is the ``three-arm probability to distance $n$.'' This implies that the ratio of this length to the length of the lowest crossing is bounded by an inverse power of $n$ with high probability. In the case of site percolation on the triangular lattice, we obtain a strict upper bound for the exponent of $4/3$.

The proof builds on the strategy developed in our previous paper \cite{DHSchemical1}, but with a new iterative scheme, and a new large deviation inequality for events in annuli conditional on arm events, which may be of independent interest.
\end{abstract}

%{\color{green}
%\section{General comments}
%\begin{enumerate}
%\item we should re-read arm separation from paper with Artem.
%\end{enumerate}
%}

\section{Introduction}
In this paper, we study the volume of crossing paths of a square $[-n,n]^2$ in two-dimensional critical Bernoulli bond percolation. We show that, conditioned on the existence of a horizontal crossing path, there exists with high probability a path whose volume is smaller than that of the lowest crossing by a factor of the form $n^{-\delta}$ for some $\delta>0$. 

\begin{thm}\label{thm: main}
Consider critical bond percolation on the edges of the $\ell^\infty$ box $[-n,n]^2\cap \mathbb{Z}^2$. Let $H_n$ be the event that there exists a horizontal open crossing of $[-n,n]^2$, and on $H_n$, let $l_n$ be the lowest open horizontal crossing. Finally, let $L_n=\# l_n$ and $S_n$ be the least number of edges of any open horizontal crossing. Then there is a $\delta>0$ and a constant $C>0$ such that
\begin{equation}
\mathbf{E}[S_n\mid H_n]\le Cn^{-\delta}\mathbf{E}[L_n|H_n] \quad \text{for all }n.
\end{equation}
\end{thm}

The minimal number of edges of any horizontal open crossing of a box is called the \underline{chemical} \underline{distance} between the left and right sides of the box. This terminology appears to originate in the physics literature, where the intrinsic distance in the graph defined by large critical percolation clusters has been studied extensively \cite{edwardskerstein, grassberger, havlin0, havlin1,hermann-stanley0,hermann-stanley,zydz}. An early reference is \cite{havlin1}, where the authors credit the physicist S. Alexander for introducing them to the term ``chemical distance.'' A common assumption in this literature is the existence of a scaling exponent $d_{\mathrm{min}}$ such that
\begin{equation}
\mathbf{E}[S_n\mid H_n]\sim n^{d_{\mathrm{min}}},
\end{equation}
where the precise meaning of $\sim$ remains to be determined. Unlike for other critical exponents in percolation, there is not even a generally accepted prediction for the exact value of $d_{\mathrm{min}}$. The existence and determination of an exponent for the chemical distance in any two-dimensional short-range critical percolation model is thus far out of reach of current methods. In particular, as noted by O. Schramm in \cite{schramm}, the chemical distance is not likely to be accessible to SLE methods. For long-range models and for correlated fields, on the other hand, there has been much recent progress; see \cite{Biskup,CernyPopov,DingSly,DRS}, and also \cite{DingLi}, where it is stated that ``it is a major challenge to compute the exponent on the chemical distance ... for critical planar percolation.'' Apart from its mathematical appeal, further progress on the chemical distance is a significant obstacle to analyzing random walks on low-dimensional critical percolation clusters (the last progress being by Kesten \cite{KestenRW} in '86) and testing the validity of the celebrated Alexander-Orbach conjecture \cite{AlexanderOrbach}.

It is known that the chemical distance in percolation clusters behaves linearly in the supercritical phase, when $p>p_c$ \cite{antalpisztora, grimmettmarstrand}. The same is true in the subcritical phase. Indeed, we have:
\begin{equation}\label{eqn: rebeccah}
\mathbf{P}_p\bigg(\mathrm{dist}_{\mathrm{chem}}(x,y)\ge \lambda |x-y| ~\bigg|~ x\leftrightarrow y\bigg)\le Ce^{-\lambda c_0|x-y|}, \quad p<p_c
\end{equation}
where $\mathrm{dist}_{\mathrm{chem}}(x,y)$ is the chemical distance between the sites $x$ and $y$ in $\mathbb{Z}^d$ and $\mathbf{P}_p(\cdot \mid x\leftrightarrow y)$ denotes the Bernoulli percolation measure with density $p$, conditioned on the event that $x$ and $y$ are connected by an open path.
This follows easily from exponential decay of the cluster volume \cite{aizenmannewman}. 
 
 In critical percolation, connected paths are expected to be \underline{tortuous} in the sense of \cite{aizenmanburchard, kestenzhang}; that is, they are asymptotically of dimension $>1$. In high dimensions, precise estimates are known, and macroscopic connecting paths have dimension 2 \cite{kozmanachmias1, kozmanachmias2, vdhs}. These estimates ultimately depend on results obtained using the lace expansion. See \cite{vdh-heyden} for a good treatment of such high-dimensional results, as well as further references.

In the low-dimensional, critical case, the chemical distance is not well understood, even at the physics level of rigor. The main result of this paper is the first nontrivial upper bound on $d_{\text{min}}$ which, combined with those of Aizenman-Burchard \cite{aizenmanburchard}, implies that for some $\delta>0$,
\begin{equation}\label{eqn: gideon}
  n^{1+\delta}\le \mathbf{E}[S_n \mid H_n] \le n^{-\delta}n^2\pi_3(n).
\end{equation}  
%where $S_n$ is the chemical distance between the vertical sides of the box. 
In site percolation on the triangular lattice, the right side of the inequality is bounded by $n^{1+s}$ for some $s<1/3$. In \cite{kestenzhang}, H. Kesten and Y. Zhang asked whether $S_n=o(L_n)$ with high probability. We answered this question affirmatively in \cite{DHSchemical1}. The possibility of the stronger inequality on the right side of \eqref{eqn: gideon} holding was also mentioned in \cite{kestenzhang}. It appears to have been expected by experts to be correct, but there is no simple, convincing heuristic for this expectation, and even no obvious reason to believe that there are crossings of different dimensions. Indeed, for large $d$, the chemical distance exponent is 2, and this coincides with the exponent for the expected total number of points on all self-avoiding open paths between two vertices that are conditioned to be connected to each other.

Our strategy builds on that in our previous paper \cite{DHSchemical1}. The key idea introduced in that paper was to construct local modifications around an edge $e$ which implied the existence of a shortcut path around $e$, conditional on $e\in l_n$, rather than to attempt to construct modifications after conditioning on $l_n$ itself. The latter point is essential; given the conditional independence of the region above the lowest crossing, a natural idea is to try to construct shortcuts around the lowest crossing in this ``unexplored'' region, conditional on $l_n$. This type of approach is doomed to failure. The roughness of the lowest crossing prevents the use of the usual volume estimates based on arm exponents, making it difficult to control the size of potential shortcuts effectively.

To improve on the bounds from \cite{DHSchemical1}, one would hope to build shortcut paths on other shortcuts, saving length on those paths that are already shorter than portions of the lowest crossing, in an inductive manner. The main difficulty with this approach is that it is not clear how to manipulate the shortest crossing; we only have information on the lowest crossing. The idea at the heart of our proof is, instead of placing shortcuts on other shortcuts, to perform an iteration on the expected lengths of shortcuts. Roughly speaking, if one can produce paths on a certain scale which have a savings over the lowest crossing, then on larger scales, one can build paths using these shortcuts in places where the lowest crossing is abnormally long. This in turn gives a larger improvement on the higher scale. 
%We use a modified definition of the events in \cite{DHSchemical1} (see also our note \cite{DHSchemical2}) and implement an iteration method to obtain improved bounds on higher and higher scales, until we can construct a path of expected length at most $n^{-\delta}n^2\pi_3(n)$. 
The main iterative result (for open paths in ``U-shaped regions'') appears in Section~\ref{sec: iteration} as Proposition~\ref{prop: main}, and we quickly derive Theorem~\ref{thm: main} from it in Section~\ref{sec: main_proof}. 
%This iteration on the length of the shortcuts is the central contribution of the present paper. 
A more detailed outline of the proof appears in the next section.

An important tool in our proof is Theorem~\ref{thm: concentration}, in Section~\ref{sec: concentration}, which is a new large deviation bound for sequences of events in disjoint annuli conditional on arm events. See the discussion in Step 2 of the proof sketch in the next section. We believe this bound should be useful for other problems.

The result presented here involves intricate gluing constructions using the Russo-Seymour-Welsh and generalized FKG inequalities. Given a description of the required connections, the details of such constructions are standard. To limit the length of this paper and focus on the original aspects of the proofs, we omit such technical details. We also frequently refer to \cite{DHSchemical1} for proofs of technical results which are similar to those appearing in that paper.
\section{Outline of the proof}
We begin by outlining the proof. In this section and the rest of the paper, given an edge $e$ and $L>0$, $B(e,L)$ denotes the box of side length $L$ centered at the lower-left endpoint of $e$. Theorem~\ref{thm: main} is a consequence of an iterative bound given in Proposition~\ref{prop: main}, so we sketch the idea for the latter's proof.

This outline splits into two parts: steps 1 - 3 summarize the construction of shortcut paths around portions of the lowest crossing $l_n$. These shortcuts are used to build a path $\sigma$ which improves on $l_n$ by a constant factor: it satisfies the bound in \eqref{eqn: tree}. Steps 4 - 5 describe the iterative procedure used to make improvements on open paths $\ell_k$ in U-shaped regions. Roughly speaking, if one can construct open paths on scale $2^k$ which improve on $\ell_k$ by a constant factor (see \eqref{eq: assume_sketch}), then for $m \geq k+C$, one can use these paths, with additional savings, to improve on $\ell_m$ by a smaller constant factor (see \eqref{eq: result_sketch}).
\begin{enumerate}
\item[Step 1.] \textbf{Construction of shortcuts.}
Given $\epsilon>0$, and an edge $e \in B(n)$, we define an event $E_k(e)$ depending on $B(e,2^K)\setminus B(e,2^k)$, with
\begin{equation}
K= k+\lfloor \log \frac{1}{\epsilon}\rfloor,
\end{equation}
such that 
\begin{equation}\label{eqn: intro-prob-lower-bound}
\mathbf{P}( E_k(e)\mid e\in l_n)\ge c\epsilon^4,
\end{equation}
for some $c>0$ and such that the occurrence of $E_k(e)$ implies the existence of an open arc $r\subset  B(e,3\cdot 2^k)$ with endpoints $u(e)$ and $v(e)$ on the lowest crossing $l_n$ of $[-n,n]^2$ and otherwise not intersecting it. Moreover, letting $\tau=\tau(r)$ be the portion of $l_n$ between $u(e)$ and $v(e)$, we have $e\in \tau$ and 
\[\frac{\# r}{\# \tau}\le \epsilon.\]
See Figure \ref{fig: shortcut}. If $E_k(e)$ occurs and $e \in l_n$, there is a shortcut $r$ on scale $2^k$ around edges of the lowest crossing which saves at least $(1/\epsilon)\cdot \#r$ edges.
The open arc $r$ is constructed in such a way that the shortcuts $r$, $r'$ resulting from the occurrence of $E_k(e)$, $E_l(e')$, are either nested or disjoint.

The definition of $E_k(e)$ appears in Section \ref{sec: Ek}.

\begin{figure}
\centering
\scalebox{0.80}{\includegraphics[trim={0 0 0 0},clip]{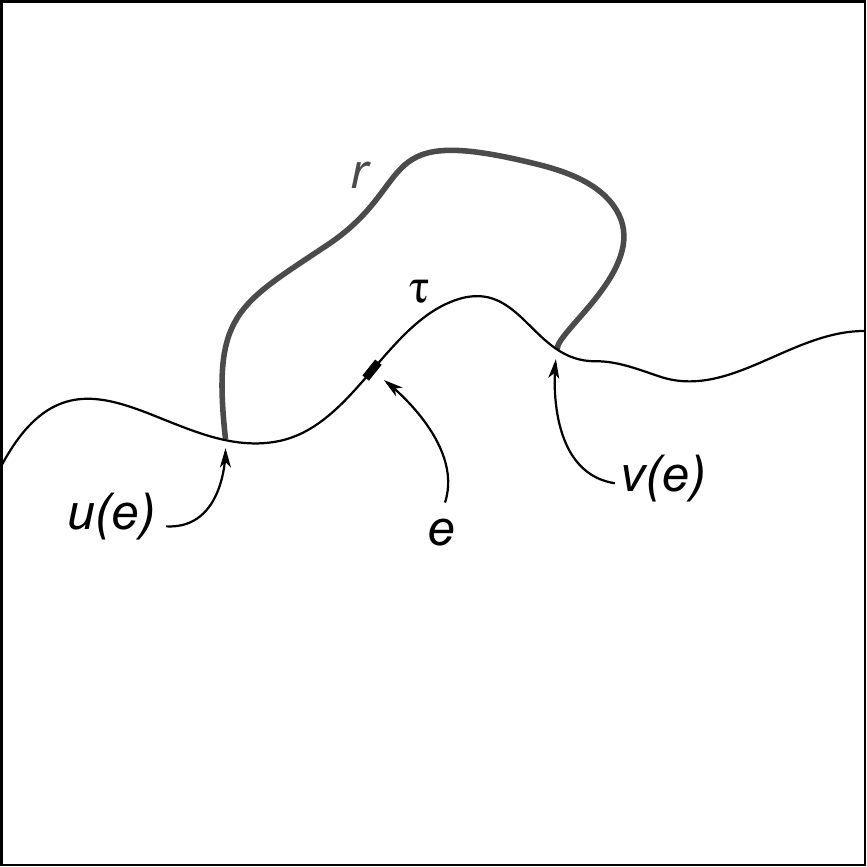}}
\caption{The topological plan of a shortcut. The outer box represents $B(n)=[-n,n]^2$. $\tau$ is a segment of the lowest crossing $l_n$, containing the edge $e$, with endpoints $u(e)$ and $v(e)$. The shortcut $r$, represented in grey, lies in the region above $l_n$, with endpoints $u(e)$ and $v(e)$. It bypasses the edge $e$.}
\label{fig: shortcut}
\end{figure}

\item[Step 2.] \textbf{Probability bound on shortcuts.} For each edge $e$, we define $\mathcal{S}(e)$ to be the collection of all shortcut paths around $e$ arising from occurrence of an event $E_k(e')$ for some $k$, and some $e'\in l_n$. Using the lower bound \eqref{eqn: intro-prob-lower-bound}, we show in Section \ref{sec: ldetouredpath} that if $\mathrm{dist}(e,\partial B(n))\ge d$, then $\mathcal{S}(e)=\emptyset$ implies that no events $E_j(e)$ occur for $j=1,\ldots, C\log d$ and so
\[
\mathbf{P}( \mathcal{S}(e)=\emptyset\mid e\in l_n)\le Cd^{-\frac{\epsilon^4 c}{\log \frac{1}{\epsilon}}}.
\]
The form of the right side follows from a large deviation bound conditional on a three-arm event from Section~\ref{sec: concentration} (developed using tools from our recent study of arm events in invasion percolation \cite{DHSinvasionarms}) that allows us to roughly decouple $E_k(e)$ and $E_j(e)$ on the event $e \in l_n$ so long as $|k-j| \geq C\log \frac{1}{\epsilon}$. Note that in our previous work \cite{DHSchemical1}, we were only able to obtain a weaker probability bound of the form\footnote{The estimate stated here does not appear in \cite{DHSchemical1}, but the method presented there can be quantified to obtain it. See the note \cite{DHSchemical2}.}
\[\mathbf{P}( \mathcal{S}(e)=\emptyset\mid e\in l_n) = o(1/\log d), \quad d \rightarrow \infty.\]

\item[Step 3.] \textbf{Construction of shorter crossing.}
Forming an arc $\sigma$ from a maximal collection of shortcuts and the remaining edges of $l_n$ with no shortcuts around them, we find (a special case of equation \eqref{eqn: central}):
\begin{equation}\label{eqn: tree}
\mathbf{E}[\# \sigma \mid H_n ]\le (\epsilon+ Cn^{-\frac{c\epsilon^4}{\log \frac{1}{\epsilon}}}) \cdot \mathbf{E}[\# l_n\mid H_n].
\end{equation}
The term  $\epsilon \mathbf{E}[\# l_n\mid H_n ]$ in \eqref{eqn: tree} is the contribution from the shortcuts, and the term of the form $n^{-c} \mathbf{E}[\# l_n\mid H_n]$ comes from estimating the expected  volume of the edges of the lowest crossing with no shortcut around them.
\item[Step 4.] \textbf{Iteration in U-shaped regions: initial step.} The shortcuts constructed in Steps 1-3 are contained in ``U-shaped''  regions of the form shown in Figure \ref{fig: ushaped} attached to the lowest crossing. We repeat the previous construction in a U-shaped region, conditional on the event $E_k'$ that there exist two five-arm points in the boxes $B_1$, $B_2$, with a closed and an open arc connecting these two points.

Denoting the outermost open arc between the five-arm points by $\ell_k$, one begins with an initial estimate in \eqref{eqn: initial} (see \cite{DHSchemical1} for similar bounds):
\[\mathbf{E}[\#\ell_k \mid E_k']\le C2^{2k}\pi_3(2^k).\]
The first step of the iteration uses the construction that led to the estimate \eqref{eqn: tree}. Inside the U-shaped region, we define a path $\sigma$ joining the two five-arm points and show in Section~\ref{sec: case_one} that its length is bounded by
\begin{align*}
\mathbf{E}[\# \sigma \mid E_k'] &\le (\epsilon + 2^{-c\frac{\epsilon^4 k}{\log \frac{1}{\epsilon}}})2^{2k}\pi_3(2^k)\\
&\leq C\epsilon^{1/2} 2^{2k}\pi_3(2^k),
\end{align*}
whenever $k$ is at least a constant $s_1$ depending on $\epsilon$ (see \eqref{eqn: getoor}).
\item[Step 5.] {\bf Iteration in U-shaped regions: inductive step.} In this step, we iterate the construction from step 4 on a large scale to improve on shortcuts from lower scales. This procedure is one of innovations in the current paper and is summarized in the central inequality \eqref{eqn: central} of Proposition~\ref{prop: central}. That inequality relates the savings in length on one scale to those on lower scales.

More precisely for the $i$-th step of the iteration, in Proposition~\ref{prop: central}, we begin with initial estimates (see \eqref{eqn: assumed-bds})
\begin{equation}\label{eq: assume_sketch}
\mathbf{E}[\#{\mathfrak{s}}_l \mid E_l'] \leq \delta_l(i) 2^{2l} \pi_3(2^l)
\end{equation}
for $l\geq 1$ and parameters $\delta_l(i)>0$. (From step 4, one can take $\delta_l(1)$ a constant for $l \leq s_1$ and $C\epsilon^{1/2}$ for $l \geq s_1$.) Here, $\mathfrak{s}_l$ is an open path connecting the two five-arm points with the minimal number of edges. We then use a version of the construction of step 4 described in Section~\ref{sec: construction} to build a path $\sigma$ out of shortcuts (saving $\kappa_l(i):=\epsilon \cdot \delta_l(i)$ on scale $l$) connecting five-arm points of a U-shaped region on scale $k$ whose expected length is bounded by the right side of \eqref{eqn: central}. In Proposition~\ref{prop: general-iteration} of Section~\ref{sec: case_two}, we bound this right side to show
\begin{equation}\label{eq: result_sketch}
\mathbf{E}[ \#\mathfrak{s}_{k} \mid E_k']\le \delta_k(i+1)2^{2k}\pi_3(2^k)
\end{equation}
for $k \geq 1$ and parameters $\delta_k(i+1)$ which can roughly be taken as
\[
\delta_k(i+1) \sim C'\epsilon^{1/2} \delta_{k-C''}(i),
\]
where $C'$ is independent of $\epsilon$ and $C''$ has order $\epsilon^{-4}(\log \frac{1}{\epsilon})^2$. Equation \eqref{eq: result_sketch} along with these values of $\delta_k(i+1)$ states that if we move up $\epsilon^{-4}$ scales, we accumulate an additional savings of $C'\epsilon^{1/2}$. This is sufficient to conclude the induction for the general bound of Proposition~\ref{prop: main}: for $2^k \geq (C\epsilon^{-4}(\log \frac{1}{\epsilon})^2)^L$ and $L \geq 1$,
\[
\mathbf{E}[\#\mathfrak{s}_k \mid E_k'] \leq (C'\epsilon^{1/2})^L 2^{2k} \pi_3(2^k).
\]

\end{enumerate}

\section{Notations}
Throughout this paper, we consider the square lattice $\mathbb{Z}^2$, viewed as a graph with edges between nearest-neighbor vertices. We denote the set of edges by $\mathcal{E}^2$. The critical bond percolation measure $\mathbf{P}$ is the product measure 
\[\mathbf{P}= \prod_{e\in\mathcal{E}^2}\frac{1}{2}(\delta_0+\delta_1)\]
on $\Omega= \{0,1\}^{\mathcal{E}^2}$, with the product sigma-algebra. For an edge $e\in \mathcal{E}$, the translation of $e=\{v_1,v_2\}$ by a vertex $v\in \mathbb{Z}^2$ is
\[\tau_v e = \{v_1+v, v_2+v\}.\]
For $\omega \in \Omega$, the translation $\tau_v \omega$ is defined by
\[(\tau_v \omega)_e= \omega_{\{v_1+v,v_2+v\}}\]
for each edge $e$.
For an event $E\subset \Omega$, we define the event translated by $-v$, $\tau_{-v} E$, by
\[\omega\in E \iff \tau_v \omega \in \tau_{-v} E.\]

A lattice \underline{path} is a sequence of vertices and edges $v_0$, $e_1$, $v_1$, $\ldots$, $e_N$, $v_N$ such that $\|v_{k-1}-v_k\|_1=1$ and $e_k=\{v_{k-1},v_k\}$. A path is called a \underline{circuit} if $v_0=v_N$. A path is called \underline{vertex self-avoiding} if $v_i=v_j$ implies $i=j$. A path (or circuit) is said to be open if all its edges are open ($\omega(e_i)=1$ for $i=1,\ldots,N$). A circuit is said to be \underline{open with $k$ defects} if all but $k$ edges are on the circuit are open.

The coordinate vectors $\mathbf{e}_1$, $\mathbf{e}_2$ are
\[ \mathbf{e}_1 = (1,0),\  \mathbf{e}_2 = (0,1).\]
The \underline{dual lattice} $(\mathbb{Z}^2)^*$ is
\[(\mathbb{Z}^2)^*= \mathbb{Z}^2 + \frac{1}{2}(\mathbf{e}_1+\mathbf{e}_2).\]
To each edge $e\in \mathcal{E}^2$, we associate a \underline{dual edge} $e^*$, the edge of $(\mathcal{E}^2)^*$ which shares a midpoint with $e$. For a configuration $\omega\in \Omega$, the dual configuration $\omega^*$ is defined by $\omega^*(e^*)=\omega(e)$. A dual path is a path made of dual vertices and edges. The definitions of circuit and circuit with defects extend to the case of dual paths in a straightforward way.

We will frequently refer to the three-arm event $A_3(n)$ that
\begin{enumerate}
\item The edge $\{0, \mathbf{e}_1\}$ is connected to $\partial B(n)$ by two open vertex-disjoint paths,
\item $(1/2)(\mathbf{e}_1-\mathbf{e}_2)$ is connected to $\partial B(n)^*$ by a closed dual path.
\end{enumerate}
For $v\in \mathbb{Z}^2$, $A_3(v,n)$ denotes the event $A_3(n)$ translated by $v$. We also consider the three-arm event centered at an edge $e=\{v_1,v_2\}$, characterized by the conditions
\begin{enumerate}
\item $e$ is connected to $\partial B(e,n)$ by two vertex-disjoint open paths,
\item The dual edge $e^*$ is connected to $\partial B(e,n)^*$ by a closed dual path.
\end{enumerate}
The probability of the three-arm event is denoted by
\[\pi_3(n):= \mathbf{P}(A_3(n)).\]
A fact concerning $\pi_3(n)$ we will use several times is the existence of a $\beta= 1-\gamma$ for some $\gamma>0$ can be chosen such that
\begin{equation}\label{eqn: leslie}
\frac{\pi_3(2^d)}{\pi_3(2^L)}\le C_52^{\beta(L-d)}, \quad d\le L,
\end{equation}
for some $C_5\ge 1$. See \cite[Lemma 2.1]{DHSchemical1}.

Throughout the paper, the usual notation for the logarithm $\log$, is reserved for the logarithm in base 2; thus in our notation
\[2^{\log x}:=2^{\log_2 x}=x\]
for all $x\in \mathbb{R}$. The nonnumbered constants $C,C'$, and so on, will represent possibly different numbers from line to line.

\section{Definition of $E_k$}\label{sec: Ek}
Suppose the event $H_n$ that there exists a horizontal open crossing of $[-n,n]^2$ occurs. Any vertex self-avoiding open path connecting the vertical sides of $[-n,n]^2$ corresponds to a Jordan arc separating the top side $[-n,n]\times \{n\}$ from the bottom side $[-n,n]\times \{-n\}$. $l_n$ is the vertex self-avoiding horizontal open crossing path such that the closed region  $\mathcal{B}(l_n)$ of $[-n,n]^2$ below and including $l_n$ is minimal.

A fact we will use very frequently, in various forms, is that  an edge $e$ is in the lowest crossing $l_n$ if and only if (a) it is open, (b) there are two vertex-disjoint open paths connecting $e$ to the left and right sides of $B(n)$, and (c) there is a closed dual path connecting $e^*$ to the bottom of $B(n)$. Using this, one can show that there are constants $c,C$ such that if $e\in B(n)$ is an edge with $\mathrm{dist}(e,\partial B(n))=d$, then 
\begin{equation}\label{eqn: 3-arm-bdy}
  c\pi_3(d)\pi_2(d,n) \le \mathbf{P}(e\in l_n \mid H_n) \le C\pi_3(d),
\end{equation}
where $\pi_k(d,n)$ is the ``$k$-arm'' probability corresponding to crossings of an annulus $B(n) \setminus B(d)$. This estimate was already used extensively in our previous paper \cite{DHSchemical1}. See for example Lemma 5.3 there. A similar claim and estimate holds for the probability that an edge $e$ belongs to other extremal crossing paths, such as the innermost circuit in a macroscopic annulus. See \cite{DHSchemical1} again.

\begin{df}[$\kappa$-shortcuts] \label{def: shortcuts} For an edge $e\in l_n$, the set $\mathcal{S}(e,\kappa)$ of $\kappa$-shortcuts around $e$ is defined as the set of vertex self-avoiding open paths $r$ with vertices $w_0,\ldots, w_M$ such that
\begin{enumerate}
\item for $i=1,\ldots M-1$, $w_i\in [-n,n]^2\setminus \mathcal{B}(l_n)$,
\item the edges $\{w_0, w_0 + \mathbf{e}_1\}$, $\{w_0-\mathbf{e}_1,w_0\}$, $\{w_M,w_M+\mathbf{e}_1\}$, and $\{w_M-\mathbf{e}_1,w_M\}$ are in $l_n$ and $w_1 = w_0+\mathbf{e}_2$, $w_{M-1}= w_M+\mathbf{e}_2$.
\item writing $\tau$ for the subpath of $l_n$ from $w_0$ to $w_M$, $\tau$ contains $e$, and the path $r\cup \tau$ is an open circuit in $[-n,n]^2$,
\item The points $w_0 + (1/2)(-\mathbf{e}_1+\mathbf{e}_2)$ and $w_M + (1/2)(\mathbf{e}_1 + \mathbf{e}_2)$ are connected by a dual closed vertex self-avoiding path $\mathfrak{c}$, whose first and last edges are vertical (translates of $\{0,\mathbf{e}_2\}$), and which lies in $[-n,n]^2\setminus \mathcal{B}(l_n)$.
\item $\#r\leq \kappa \# \tau$.
\end{enumerate}
\end{df}

For $0<\epsilon < 1$, we  define the annulus
\[A(2^k,2^K):=[-2^K,2^K]^2\setminus [-2^k,2^k]^2,\]
where 
\begin{equation}\label{eqn: K-def}
K= k+ \lfloor \log \frac{1}{\epsilon}\rfloor.
\end{equation}
For $\delta>0$, we define an event $E_k = E_k(\epsilon,\delta)$ depending only on the edges in the annulus $A(2^k,2^K)$ which implies the existence of a $\delta\epsilon$-shortcut around $e$ when $e\in l_n$. The next subsection contains a precise description of $E_k$. It involves a large number of connections, and appears in equation \eqref{eqn: Ek-def}, following Proposition \ref{prop: if-Exp}.
 The event is illustrated in Figures \ref{fig: main-Ek} and \ref{fig: small}. We encourage the reader to study these figures. The important features include:
\begin{itemize}
\item An open arc (shortcut), whose length is of order at most $\delta 2^{2k}\pi_3(2^k)$, connecting two arms emanating from the 3-arm edge $e$. This arc lies inside a box of side length $3\cdot 2^k$ centered at $e$, and is depicted as the top (solid) arc in Figure~\ref{fig: small}.
\item A path with length of order at least $2^{2K}\pi_3(2^K)$, whose edges necessarily lie on the lowest crossing if $e$ does. This path is depicted as the pendulous curve in Figure~\ref{fig: main-Ek}.
\end{itemize}

\begin{figure}
\centering
\scalebox{0.55}{\includegraphics[trim={0 0 0 0},clip]{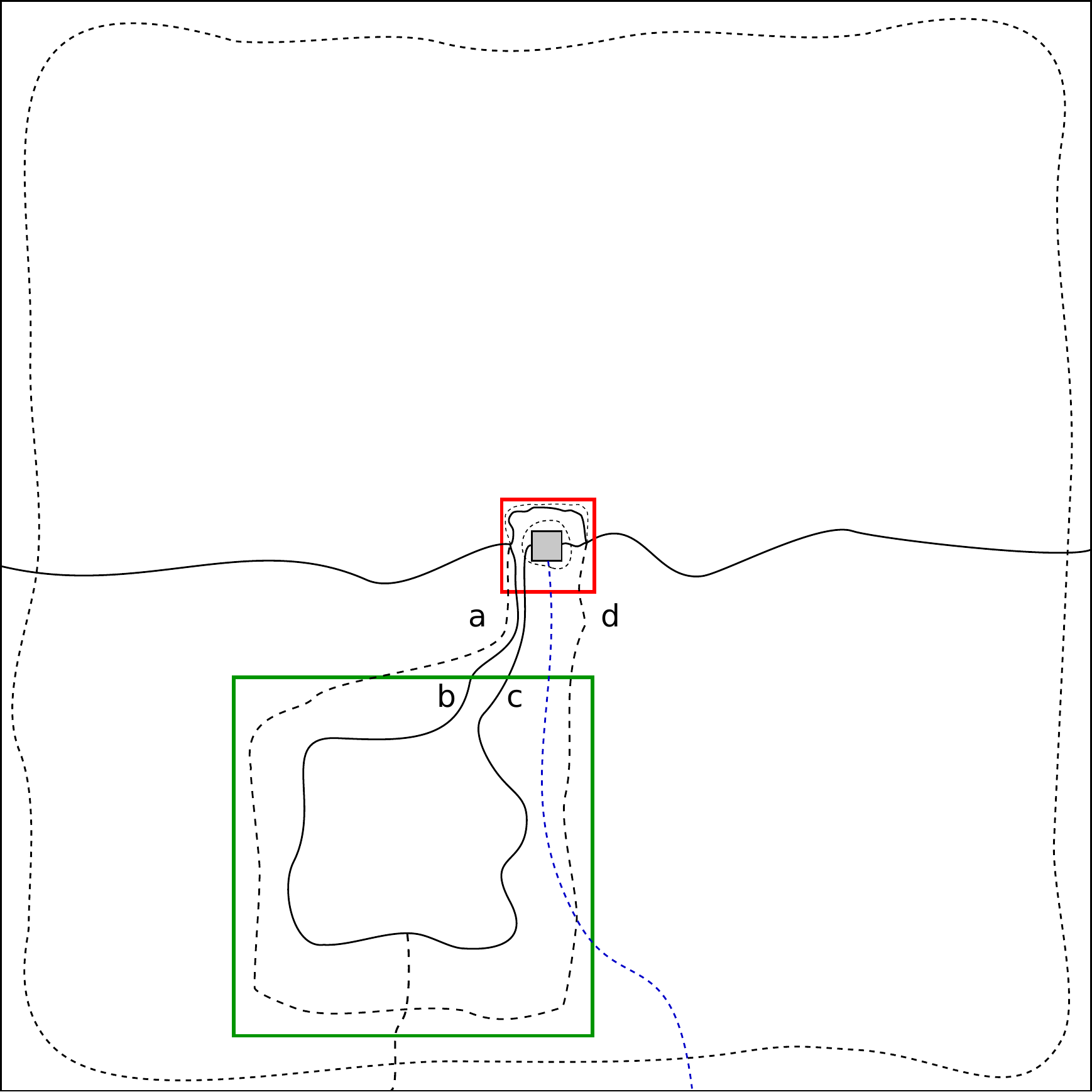}}
\caption{The event $E_k$. The outermost square represents the boundary of the box $[- 2^K, 2^K]^2$. The box with red boundary is $[-3 \cdot 2^k, 3 \cdot 2^k]^2$ This box contains a shortcut which bypasses the lowest crossing. Details of the construction inside the red box appear in Figures \ref{fig: small} and \ref{fig: small-2}. 
The blue path is not part of the definition of $E_k$. Rather, it is present if the grey box contains an edge of the lowest path of a larger box. In this case, all three-arm points in the green box with a closed dual arm to the bottom of that box also lie on the lowest path. The green box has size of order $\mathrm{const.}\times 2^K$.
The asymmetric placement of the green box serves to make it clear that both the shortcut and detoured paths in $E_k$ (the blue paths in Figures \ref{fig: small-2} and \ref{fig: details}, respectively) are contained in the box of side-length $2^{K}+3\cdot 2^k$ whose lower left corner coincides with that of $[-2^K,2^K]^2$. This will be essential for the iteration scheme in Section \ref{sec: iteration}.}
\label{fig: main-Ek}
\end{figure}

\begin{figure}
\centering
\scalebox{0.6}{\includegraphics[trim={0 0 0 0},clip]{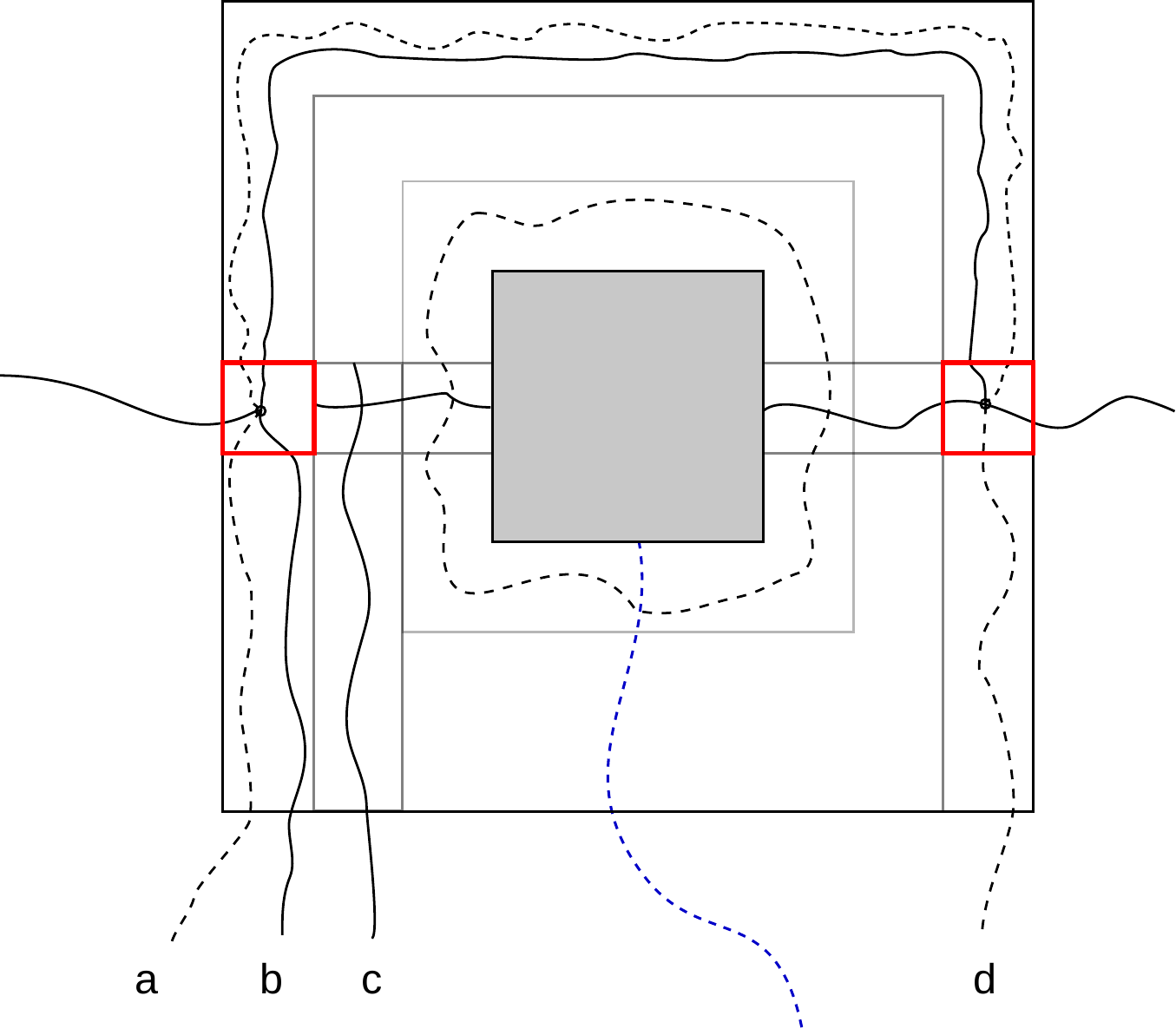}}
\caption{The inner box from Figure \ref{fig: main-Ek}: the grey area represents the box $[-2^k,2^k]^2$. The outer square is the boundary of the box $[-3\cdot 2^k, 3\cdot 2^k]^2$. The paths labeled $a$ and $d$ in the lower part of the figure are part of a dual closed path enclosing $[-2^k, 2^k]^2$ (see Figure \ref{fig: main-Ek}). The paths $b$ and $c$ are part of an arc containing on the order of $2^{2K} \pi_3(2^K)$ points, which all lie on the lowest crossing $l_n$ if the box $[-2^k,2^k]^2$ contains an edge on the lowest crossing. The two red boxes ($B_1$ on the left, $B_2$ on the right) each contain a five-arm point. These five-arm points $\star_1$ and $\star_2$ are connected by an open arc, the shortcut that bypasses the path between $b$ and $c$. The latter is also depicted in blue in Figure \ref{fig: details}.}
\label{fig: small}
\end{figure}

We denote by $E_k(e) = E_k(e,\epsilon,\delta)$ the event $\tau_{-e_x}E_k(\epsilon,\delta)$, that is, the event that $E_k$ occurs in the configuration $(\omega_{e+e_x})_{e\in\mathcal{E}(\mathbb{Z}^2)}$ translated by the coordinates of the lower-left endpoint $e_x$ of the edge $e$.

Two properties of $E_k(e)$ which will be crucial for the rest of the proof are:
\begin{enumerate}
\item If $E_k(e)$ occurs for some $k$ and $e$ lies on $ l_n$, then $\mathcal{S}(e,\delta\epsilon)\neq \emptyset$ . (See Proposition \ref{prop: implies-S}.)
\item We have the following lower bound for the probability of $E_k$, assuming a bound of the form \eqref{eqn: Esk-def}, expressing a length gain of $\delta$: there is a constant $C>0$ such that
\[\mathbf{P}(E_k\mid A_3(2^R))\ge C\epsilon^4.\]
for any $R\ge K$. (See Proposition \ref{prop: e4-lwr} and \eqref{eqn: true-triple-c}.)
\end{enumerate}

\subsection{Connections in $E_k$}
\label{sec: Ek}
In this section, we enumerate all the conditions for the occurrence of the event $E_k$. We first define an auxiliary event $E'_k$, which will contain most of the conditions defining $E_k$.

\subsubsection{Inside the box $[-3\cdot 2^k, 3\cdot 2^k]^2$.}\label{sec: insideek}
All connections described below remain in the annulus $[-2^K,2^K]^2\setminus [-2^k,2^k]^2$, so that the events are different for different values of $k$. First, we have a number of conditions depending on the status of edges inside $[-3\cdot 2^k,3\cdot 2^k]^2$ (see Figure \ref{fig: small-2}).

We use the term five-arm point in the following way. The origin is a five-arm point if it has three  vertex-disjoint open paths emanating from 0, one taking the edge $\{0, \mathbf{e}_1\}$ first, one taking the edge $\{0, -\mathbf{e}_1\}$ first, and one taking the edge $\{0, \mathbf{e}_2\}$ first. The two remaining closed dual paths emanate from dual neighbors of 0, one taking the dual edge $\{(-1/2)\mathbf{e}_1+(1/2)\mathbf{e}_2, (-1/2)\mathbf{e}_1 + (3/2)\mathbf{e}_2\}$ first, and the other taking the dual edge $\{(1/2)\mathbf{e}_1-(1/2)\mathbf{e}_2, (1/2)\mathbf{e}_1-(3/2)\mathbf{e}_2\}$ first. We denote the event that the origin is a five-arm point to distance $n$ by $A_5(n)$.

\begin{figure}
\centering
\scalebox{0.8}{\includegraphics[trim={0 0 0 0},clip]{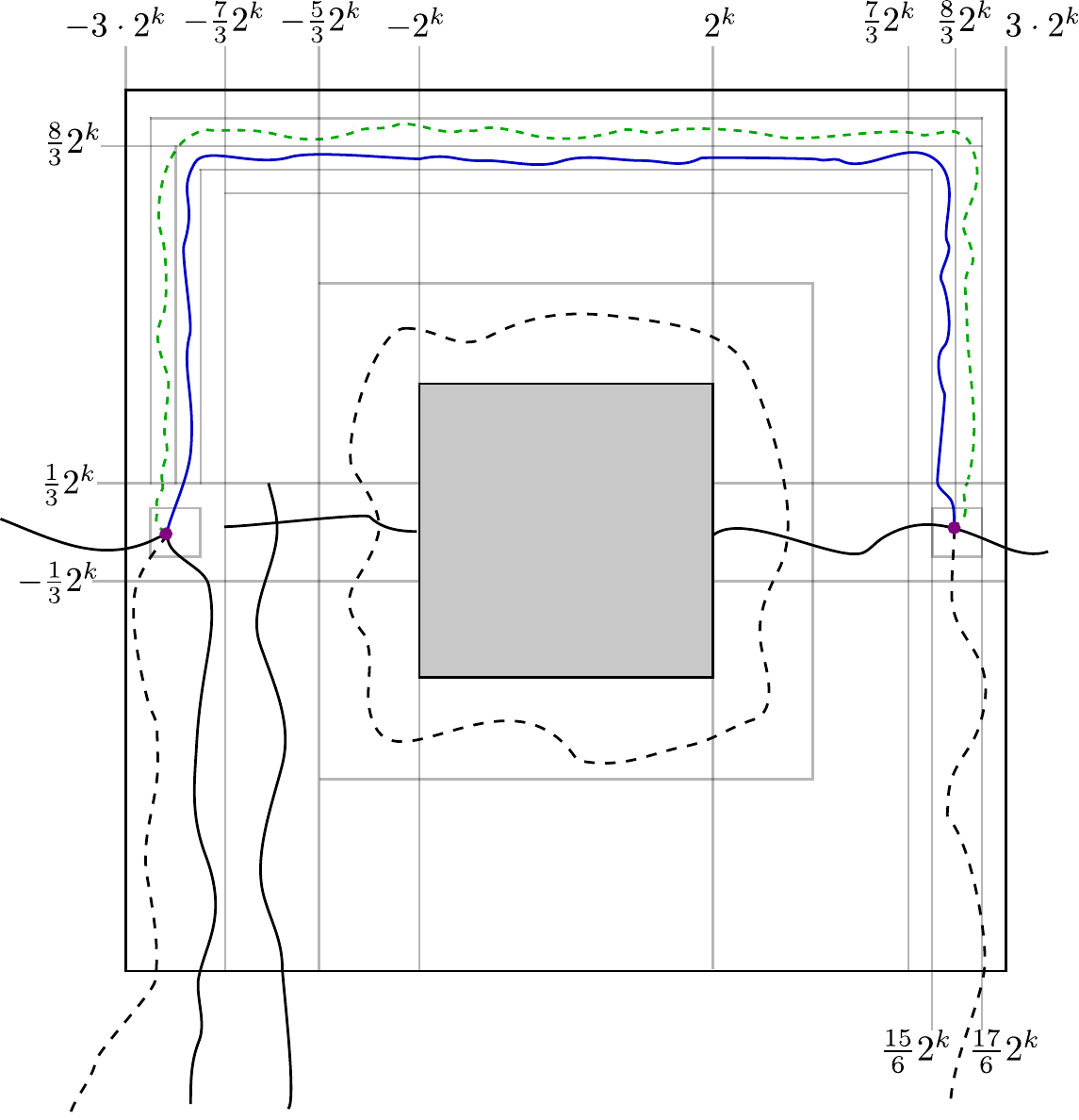}}
\caption{An illustration of the connections inside $[-3\cdot 2^k, 3\cdot 2^k]$ implied by the occurrence of $E_k$. The shortcut (item 8. in the definition of $E_k'$) appears in blue. The shielding dual closed arc (item 7. in the definition of $E_k'$) appears in green. The two five arm points, $\star_1$ in $B_1$ and $\star_2$ in $B_2$ (items 3. and 4.), are depicted by purple dots.}
\label{fig: small-2}
\end{figure}

\begin{enumerate}
\item There is a horizontal open crossing of $[2^k,3\cdot 2^k] \times [-\frac{2^k}{3}, \frac{2^k}{3}] $, and a horizontal open crossing of $[-\frac{7}{3}\cdot 2^k, -2^k]\times[-\frac{2^k}{3}, \frac{2^k}{3}]$.
\item There is a vertical open crossing of $[-\frac{7}{3}\cdot 2^k, -\frac{5}{3}\cdot 2^k]\times [-3\cdot 2^k, \frac{1}{3}\cdot 2^k]$.

\item There is a five-arm point (represented by a purple dot in Figure \ref{fig: small-2}) in the box 
\[B_1:=[-\frac{17}{6}\cdot 2^k, -\frac{15}{6}\cdot 2^k ]\times [-\frac{1}{6}\cdot 2^k, \frac{1}{6}\cdot 2^k],\]
with the following connections, in clockwise order:
\begin{enumerate}
\item a closed dual arm connected to $[-\frac{17}{6}\cdot 2^k, -\frac{8}{3}\cdot 2^k]\times \{\frac{1}{3}\cdot 2^k\}$,
\item an open arm connected to $[-\frac{8}{3}\cdot 2^k, -\frac{15}{6}\cdot 2^k]\times \{\frac{1}{3}\cdot 2^k\}$,
\item an open arm connected to  $[-\frac{8}{3}\cdot 2^k, -\frac{7}{3}\cdot 2^k]\times \{-\frac{1}{3}\cdot 2^k\}$,
\item a closed dual arm connected to $[-3\cdot 2^k, -\frac{8}{3}\cdot 2^k]\times \{-\frac{1}{3}\cdot 2^k\}$,
\item and an open arm connected to the ``left side'' of the box, $\{-3\cdot 2^k\}\times [-\frac{1}{3}\cdot 2^k, \frac{1}{3}\cdot 2^k]$.
\end{enumerate}
We denote the unique such point in $B_1$ by $\star_1$.

\item There is a five-arm point (represented by a purple dot in Figure \ref{fig: small-2}) in the box 
\[B_2:=[\frac{15}{6}\cdot 2^k, \frac{17}{6}\cdot 2^k]\times [-\frac{1}{6}\cdot 2^k, \frac{1}{6}\cdot 2^k],\]
with the following connections, in clockwise order:
\begin{enumerate}
\item an open arm connected to $[\frac{15}{6}\cdot 2^k,\frac{8}{3}\cdot 2^k]\times \{\frac{1}{3}\cdot 2^k\}$, 
\item a closed dual arm connected to $[\frac{8}{3}\cdot 2^k, \frac{17}{6}\cdot 2^k]\times \{\frac{1}{3}\cdot 2^k\}$, 
\item an open arm connected to the ``right side'' of the box $\{3\cdot 2^k\}\times [-\frac{1}{3}\cdot 2^k, \frac{1}{3}\cdot 2^k]$, 
\item a closed dual arm connected to $[\frac{7}{3}\cdot 2^k, 3\cdot 2^k]\times \{-\frac{1}{3}\cdot 2^k\}$, 
\item and an open arm connected to $\{\frac{7}{3}\cdot 2^k\}\times [-\frac{1}{3}\cdot 2^k,\frac{1}{3}\cdot 2^k]$.
\end{enumerate}
We denote the unique such  point in $B_2$ by $\star_2$.

\item There is a closed dual circuit with two open defects around the origin inside the annulus $[-\frac{5}{3}2^k,\frac{5}{3}2^k]^2\setminus [-2^k,2^k]$. One of the defects is in the box $[-\frac{5}{3}2^k, -2^k]\times [-\frac{1}{3}2^k, \frac{1}{3}2^k]$, and the other is in $[2^k,\frac{5}{3}2^k]\times [-\frac{1}{3}2^k, \frac{1}{3}2^k]$.

\item There is an open vertical crossing of $[-3\cdot 2^k, -\frac{7}{3}\cdot 2^k]\times [-3\cdot 2^k,-\frac{1}{3}\cdot 2^k]$, connected to the open arm that emanates from the five-arm point in $B_1$ and lands in $[-\frac{8}{3}\cdot 2^k, -\frac{7}{3}\cdot 2^k]\times \{-\frac{1}{3}\cdot 2^k\}$. There is a dual closed vertical crossing of $[-3\cdot 2^k, -\frac{7}{3}\cdot 2^k]\times [-3\cdot 2^k, -\frac{1}{3}\cdot 2^k]$, connected to the closed dual arm that lands in  $[-3\cdot 2^k, -\frac{8}{3}\cdot 2^k]\times \{-\frac{1}{3}\cdot 2^k\}$.

\item There is a closed dual vertical crossing of $[\frac{7}{3}\cdot 2^k, 3\cdot 2^k]\times [-3\cdot 2^k, -\frac{1}{3}\cdot 2^k]$, connected to the dual arm that lands in $[\frac{7}{3}\cdot 2^k, 3\cdot 2^k]\times \{-\frac{1}{3} \cdot 2^k\}$.

\item There is a closed dual arc (the shield, in green in Figure \ref{fig: small-2}) in the half-annulus 
\begin{equation}\label{eqn: Vktilde-def}
\tilde{V}(k):=\bigg[ [-\frac{17}{6}\cdot 2^k, \frac{17}{6}\cdot 2^k]\times [\frac{-1}{6}\cdot 2^k,\frac{17}{6}\cdot 2^k]  \bigg] \setminus (-\frac{8}{3} \cdot 2^k, \frac{8}{3} \cdot 2^k)^2
\end{equation}
connecting the closed dual paths from the two five-arm points in items 3 and 4.

\item There is an open arc (the shortcut, in blue in Figure \ref{fig: small-2}) in the region 
\begin{equation}\label{eqn: Uktilde-def}
\tilde{U}(k):=\bigg[ [-\frac{8}{3} \cdot 2^k, \frac{8}{3} \cdot 2^k]\times [-\frac{1}{6}\cdot 2^k, \frac{8}{3} \cdot 2^k]\bigg] \setminus (-\frac{15}{6} \cdot 2^k, \frac{15}{6} \cdot 2^k)^2,
\end{equation}
connecting the open paths from the two five-arm points in items 3 and 4 which land on the line $\{(x,\frac{1}{3} \cdot 2^k) : x \in \mathbb{Z}\}$.
\end{enumerate}

\subsection{The box $[-2^K,2^K]^2$ and the large detoured path}\label{sec: ldetouredpath}
The following connections occur in the box $[-2^K,2^K]^2$. Refer to Figure \ref{fig: details} for an illustration and the relevant scales.

\begin{figure}
\centering
\scalebox{0.6}{\includegraphics[trim={0 0 0 0},clip]{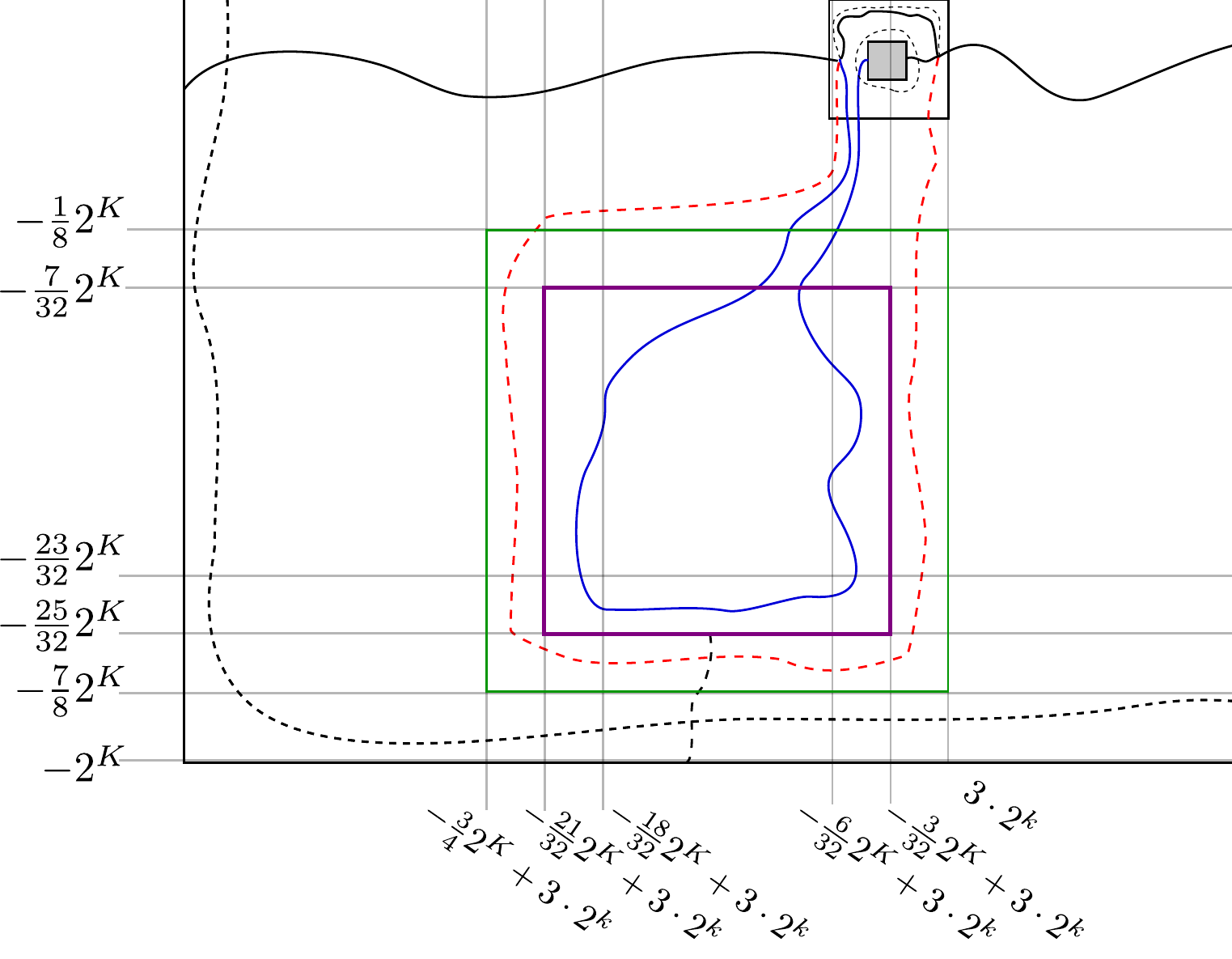}}
\caption{An illustration of the connections on scale $2^K$ implied by $E_k$. Note that this schematic is not to scale: in particular, the scale $2^k$ is much smaller than $2^K$. For example, the right edge of the purple box is in reality far to the left of $[-3\cdot 2^k, 3\cdot 2^k]^2$. }
\label{fig: details}
\end{figure}

\begin{enumerate}
\setcounter{enumi}{9}
\item There is a closed dual arc $\tau$ around $D_2$ in $D_1\setminus D_2$,
where
\[D_1:=[-\frac{3}{4}2^K+3\cdot 2^k, 3\cdot 2^k] \times [-\frac{7}{8}2^K, -\frac{1}{8}2^K].\]
is depicted in green in Figure \ref{fig: details}, and
\[D_2:=[-\frac{21}{32}2^K + 3\cdot 2^k,-\frac{3}{32}2^K+3\cdot 2^k]\times [-\frac{25}{32}2^K, -\frac{7}{32}2^K]\]
is in purple in Figure \ref{fig: details}. The path $\tau$ joins
 $[-\frac{3}{4}2^K+3\cdot 2^k, -\frac{21}{32}2^K+3\cdot 2^k]\times \{-\frac{2^K}{8}\}$ to $[-\frac{3}{32}2^K+3\cdot 2^k, 3\cdot 2^k]\times\{- \frac{2^K}{8}\}$. $\tau$ is the part of the path represented in red in Figure \ref{fig: details} that lies inside $D_1$.
\item There are two disjoint closed paths inside $[-\frac{3}{4}2^K+3\cdot 2^k,3\cdot 2^k]\times ((-\infty,0])$: one joining the endpoint of the vertical closed crossing on $[-3\cdot 2^k, -\frac{8}{3}\cdot 2^k]\times \{-3\cdot 2^k\}$ to the endpoint of the closed dual arc in the previous item on  $[-\frac{3}{4}2^K+3\cdot 2^k, -\frac{21}{32}2^K+3\cdot 2^k]\times \{-\frac{2^K}{8}\}$, the second, joining the endpoint of the vertical crossing on $[\frac{7}{3}\cdot 2^k,3\cdot 2^k]\times \{-3\cdot 2^k\}$, to the endpoint of the closed dual arc in the previous item on $[-\frac{3}{32}2^K+3\cdot 2^k, 3\cdot 2^k]\times\{- \frac{2^K}{8}\}$. The union of these two paths is represented in Figure \ref{fig: details} as the part of the red path outside of the box $D_1$.
\item There is a horizontal open crossing of the rectangular box
\begin{equation}\label{eqn: R-def}
R:=[-\frac{18}{32}\cdot 2^K+3\cdot 2^k, -\frac{6}{32}\cdot 2^K+3\cdot 2^k]\times [-\frac{25}{32}\cdot 2^K, -\frac{23}{32}\cdot 2^K].
\end{equation}
This is the part of the path appearing in blue in Figure \ref{fig: details} which lies in $R$.
\item There are two disjoint open paths contained in $[-\frac{21}{32}\cdot 2^K+ 3\cdot 2^k,-\frac{3}{32}\cdot 2^K+3\cdot 2^k]\times [-\frac{25}{32}\cdot 2^K,0]\setminus R$,
\begin{enumerate} 
\item one joining the endpoint of the open vertical crossing of $[-\frac{8}{3}\cdot 2^k, -\frac{7}{3}\cdot 2^k]\times [-3\cdot 2^k,-\frac{1}{3}\cdot 2^k]$  to the endpoint of the open crossing of $R$ (see \eqref{eqn: R-def}) on the left side of $R$, 
\item one joining the endpoint of the open vertical crossing of $[-\frac{7}{3}\cdot 2^k,-\frac{5}{3}\cdot 2^k ]\times [-3\cdot 2^k,-\frac{1}{3}\cdot 2^k]$ to  endpoint of the open crossing of $R$ on right side of $R$.
\end{enumerate}
The union of these two paths is the part of the path depicted in blue in Figure \ref{fig: details} lying outside of $R$.
\item There is dual closed vertical crossing of $[-\frac{18}{32}\cdot 2^K+3\cdot 2^k, -\frac{6}{32}\cdot 2^K+3\cdot 2^k]\times [-2^K, -\frac{25}{32}\cdot 2^K]$.
\end{enumerate}

Finally, we finish the description of the event by adding two more macroscopic conditions: 
\begin{enumerate}\setcounter{enumi}{14}
\item There is a dual closed circuit with two open defects around the origin in $[-2^K, 2^K]^2\setminus [-\frac{7}{8}\cdot 2^K,\frac{7}{8}\cdot 2^K]^2$. One of the defects is contained in $[-2^K, -\frac{7}{8}\cdot 2^K]\times [-\frac{2^K}{8},\frac{2^K}{8}]$, and the other in $[\frac{7}{8}\cdot 2^K, 2^K]\times [-\frac{2^K}{8},\frac{2^K}{8}]$.
\item There are two vertex-disjoint open arms: one from the left side $\{-3\cdot 2^k\}\times [-3\cdot 2^k, 3\cdot 2^k]$ of the box $[-3\cdot 2^k, 3\cdot 2^k]^2$ (touching the open arm from the five-arm point that lands there) to the left side of $[-2^K,2^K]^2$, the other from the right side $\{3\cdot 2^k\}\times [-3\cdot 2^k, 3\cdot 2^k]$ (touching the corresponding open arm from the five-arm point there) to the right side of $[-2^K,2^K]^2$.
\end{enumerate}
We denote by $E'_k=E'_k(\epsilon)$ the intersection of the events listed in items 1-16 above. We also let $E'_k(e,\epsilon)=\tau_{-e_x}E'_k(\epsilon)$ be the event translated by $e$.

By considering three-arm points in the rectangle $R$ (defined in \eqref{eqn: R-def}), we have the following proposition. See Proposition 5.4 in \cite{DHSchemical1} for a more detailed treatment of a similar construction.
\begin{prop}\label{prop: c0-prop}
On $E'_k$, let $N_K$ be the number of edges in $R$ connected to the open paths from item 12. by two vertex-disjoint open paths inside $R$ which moreover are connected inside $R$ by a dual closed path to the dual path in item 13. There is a constant $c_0>0$ such that for $\epsilon \in (0,1/4)$ and any $k\ge 1$, one has
\begin{equation}\label{eqn: N-lwr-bd}
\mathbf{P}(\{N_K\ge c_0 2^{2K}\pi_3(2^K) \}\cap E'_k) \ge c_0\mathbf{P}(E'_k).
\end{equation}
\begin{proof}
By the second moment method, one shows that the number of edges in a central subrectangle of $R$ with two disjoint open connections to the left and right side of $R$ and a dual closed connection to the bottom side of $R^*$ is bounded below by $c_0 2^{2K}\pi_3(2^K)$ with uniformly positive probability. By gluing constructions using the Russo-Seymour-Welsh (RSW) and generalized Fortuyn-Kasteleyn-Ginibre (FKG) inequalities, the open arms are connected to the open connections from $E'_k$ from the left and right side, and the closed arm is connected to the closed connection from item 13. 
\end{proof}
\end{prop}

On $E_k'$, let $\mathfrak{s}_k$ be the minimal length open path connecting the two five-arm points $\star_1$ and $\star_2$ in the $U$-shaped region 
\begin{equation}\label{eqn: Udef}
U(k):=\left[ [-3\cdot 2^k, 3 \cdot 2^k] \times [-\frac{1}{3} \cdot 2^k, 3 \cdot 2^k]\right] \setminus (-\frac{7}{3} \cdot 2^k, \frac{7}{3} \cdot 2^k)^2.
\end{equation}

\begin{lma}\label{prop: if-Exp}
Let $c_0$ be from Proposition~\ref{prop: c0-prop}. If for some $\epsilon \in (0,1/4)$, $\delta>0$ and $k \geq 1$ one has
\begin{equation}\label{eqn: Esk-def}
\mathbf{E}[ \#\mathfrak{s}_k \mid E_k']\le \delta 2^{2k}\pi_3(2^k),
\end{equation}
then
\begin{equation}\label{eqn: Ek-half}
\mathbf{P}(\# \mathfrak{s}_k \le 2(\delta/c_0) 2^{2k}\pi_3(2^k) \mid N_K\ge c_0 2^{2K}\pi_3(2^K), E_k')\ge 1/2.
\end{equation}
\begin{proof}
Let $\mathcal{N}=\{N_K\ge  c_0 2^{2K}\pi_3(2^K) \}$. 
We have
\begin{align}
\mathbf{E}[\# \mathfrak{s}_k \mid \mathcal{N}, E_k'] \le &~\frac{\mathbf{E}[\# \mathfrak{s}_k \mathbf{1}_{E_k'}]}{\mathbf{P}(\mathcal{N}\cap E_k')} \nonumber \\
= &~ \mathbf{E}[\# \mathfrak{s}_k \mid E_k'] \cdot \frac{\mathbf{P}(E_k')}{\mathbf{P}(\mathcal{N}\cap E_k')}.\label{eqn: kellogs}
\end{align}
By \eqref{eqn: N-lwr-bd}, the second factor is bounded above by $1/c_0$ for some constant $c_0>0$, so \eqref{eqn: kellogs} is bounded by
\[\frac{\delta}{c_0}2^{2k}\pi_3(2^k).\]
The result then follows by Markov's inequality.
\end{proof}
\end{lma}

We now define the event $E_k = E_k(\epsilon,\delta)$ as
\begin{equation}
E_k(\epsilon,\delta):= \{\# \mathfrak{s}_k  \le 2(\delta/c_0)2^{2k}\pi_3(2^k)\}\cap \{N_K\ge c_02^{2K}\pi_3(2^K)\} \cap E_k',
\end{equation}
as well as the translated event $E_k(e) = E_k(e,\epsilon,\delta)$ as
\begin{equation}
\label{eqn: Ek-def}
E_k(e,\epsilon,\delta):= \tau_{-e_x}E_k(\epsilon, \delta).
\end{equation}

The key property of $E_k$ is the following. For $e \in B(n)$, write $d = \text{dist}(e,\partial B(n))$.
\begin{prop}\label{prop: implies-S}
There is an $\epsilon_0$ such that if $\epsilon < \epsilon_0$, $\delta>0$, and $k$ satisfies $1 \leq k \leq \log d - \lfloor \log \frac{1}{\epsilon}\rfloor$,  the occurrence of
\[\{e\in l_n \}\cap E_k(e,\epsilon,\delta)\] 
implies that there is an $\epsilon\delta$-shortcut around $e$, i.e. $\mathcal{S}(e,\kappa)\neq \emptyset$ for $\kappa = \epsilon \cdot \delta$. 
\begin{proof}
  It follows from the construction of the event $E_k(e,\epsilon,\delta)$ that there a shortcut around $e$. See \cite[Sections 4.5 and 7]{DHSchemical1} for a detailed proof of a similar claim. The event $E_k$ there is defined differently, but the arguments remain essentially the same. For the path $r$, we choose a path in $\tau_{e_x}U(k)$ between the two five-arm points in items 3. and 4. of the definition of $E_k'$ above with length less than $2(\delta/c_0) 2^{2k}\pi_3(2^k)$. On the other hand because the edges found in Proposition \ref{prop: c0-prop} are on the lowest crossing, the portion $\tau$ of $l_n$ containing $e$ between the two five-arm points has total volume greater than or equal to 
\[N_K\ge c_0 2^{2K}\pi_3(2^K).\]
Thus, 
\begin{equation}
\frac{\# r}{\#\tau} \le \delta \frac{(2/c_0) 2^{2k}\pi_3(2^k)}{ c_02^{2K}\pi_3(2^K)}.
\end{equation}
Using \eqref{eqn: leslie}, we have
\[\frac{2^{2k}\pi_3(2^k)}{2^{2K}\pi_3(2^K)} \le C_5 2^{(2-\beta)(k-K)} \leq 2C_5 2^{2-\beta} \epsilon^{2-\beta},\]
where $C_5\ge 1$ is a constant, and $\beta=1-\gamma$, for $\gamma>0$. If
\begin{equation}\label{eqn: epsilon-const-1}
\epsilon^\gamma < \min \left\{ \frac{c_0^2}{8C_5 2^{2-\beta}},1/4^\gamma\right\},
\end{equation}
then we find 
\begin{equation}
\# r < (\epsilon\delta)\cdot \#\tau.
\end{equation}
\end{proof}
\end{prop}

The following proposition gives a lower bound for the probability of $E_k(e,\epsilon,\delta)$:
\begin{prop}\label{prop: e4-lwr}
There is a constant $c_2>0$ such that for all $\epsilon \in (0,1/4)$ and $k \geq 1$,
\begin{equation}\label{eqn: Ekprime-lwr}
\mathbf{P}(E_k')\ge c_2\epsilon^4.
\end{equation}
In particular, by \eqref{eqn: N-lwr-bd}, if \eqref{eqn: Esk-def} holds for some $\epsilon \in (0,1/2)$, $\delta>0$, and $k \geq 1$, then
\begin{equation}\label{eqn: triplec-lwr-bd}
\mathbf{P}(E_k(e,\epsilon,\delta))\ge \frac{c_0c_2}{2} \epsilon^4.
\end{equation}
\end{prop}
\begin{proof}
The second inequality is a combination of \eqref{eqn: N-lwr-bd}, \eqref{eqn: Ek-half} and \eqref{eqn: Ekprime-lwr}. For the first, we apply the RSW and the generalized FKG inequalities to construct all the connections in the definition of $E_k'$. The construction of the five-arm points in items 3. and 4. uses the second moment method and 
\[\mathbf{P}(A_5(n))\ge Cn^{-2},\]
where $A_5(n)$ is the event that there is a polychromatic five-arm sequence from 0 to distance $n$ (see the definition at the beginning of Section \ref{sec: insideek}). The main probability cost comes from connecting 6 arms (two closed and four open), corresponding to the connections in items 10., 12. and 16. above, and this has probability at least a constant times $\epsilon^4$:
\begin{align*}
\mathbf{P}(E_k') &\geq C \mathbf{P}(\exists~5\text{-arm points in } B_i, i=1,2) \mathbf{P}(A_6(2^k,2^K))\\
& \geq C \sum_{\star_1,\star_2} \mathbf{P}(A_5(2^k))^2 \mathbf{P}(A_6(2^k,2^K))\\
&\ge C\epsilon^4.
\end{align*}
\end{proof}

Since $E_k(e,\epsilon,\delta)$ implies in particular the existence of 3 disjoint connections (2 open, one closed) between $\partial B(e,2^k)$ and $\partial B(e,2^K)$, by a straightforward gluing argument (see \cite[Section 5.5]{DHSchemical1}), we pass from the lower bound \eqref{eqn: triplec-lwr-bd} to the following conditional bound. There is $c_4>0$ such that if \eqref{eqn: Esk-def} holds for some $\epsilon\in (0,1/2)$, $\delta>0,$ and $k \geq 1$, then for all $L \geq 1$,
\begin{equation}\label{eqn: true-triple-c}
\mathbf{P}(E_k(e,\epsilon,\delta)\mid A_3(e,2^L))\ge c_4\epsilon^4.
\end{equation}

\begin{prop}\label{prop: concentration}
There is a constant $\hat{c}$ such that if $\delta_j>0$, $j=1, \ldots, L$ is a sequence of parameters such that for some $\epsilon \in (0,1/4)$,
\begin{equation}
\mathbf{E}[\# \mathfrak{s}_j \mid E_j']\le \delta_j 2^{2j}\pi_3(2^j),
\end{equation}
then for any, $L'<L$,
\begin{equation}\label{eq: lasagna_supremo}
\mathbf{P}(\cap_{j=L'}^L E_j(e,\epsilon,\delta_j)^c \mid A_3(e,2^L))\le 2^{-\hat{c}\frac{\epsilon^4}{\log \frac{1}{\epsilon}}(L-L')}.
\end{equation}
\end{prop}
\begin{proof}
  % Since $E_j=E_j(e,\epsilon,\epsilon_j)$ depends only on the status of edges in $B(e,2^{j+\lfloor \log \frac{1}{\epsilon}\rfloor})\setminus B(e,2^j)$, an annulus of aspect ratio $\sim 1/\epsilon$, we can find a subcollection of indices $I\subset \{L', \ldots, L\}$ with $\#I \ge \frac{1}{10 \log \frac{1}{\epsilon}}(L-L')$ such that $(E_j)_{j\in I}$ are independent and $\max_{j\in I}j \le L-\lfloor \log \frac{1}{\epsilon}\rfloor$.
Putting $E_j=E_j(e,\epsilon,\delta_j)$, we have by \eqref{eqn: true-triple-c},
\[\mathbf{P}(E_j\mid A_3(e,2^L))\ge c_4\epsilon^4, \quad j,L\ge 1.\]
% The collection $(E_j)_{j\in I}$ is not independent conditionally on $A_3(e,2^L)$. However, we may use Theorem \ref{thm: concentration} with $N= 10\lfloor \log \frac{1}{\epsilon}\rfloor$ and $C_0 = c_4\epsilon^4$ to find a constant $c_5>0$ such that
Furthermore, using the notation of Theorem \ref{thm: concentration}, straightforward gluing constructions can be used to show that, by possibly lowering $c_4$, one has
\[\mathbf{P}(E_{10j+5}, \hat{\mathfrak{C}}_j \mid A_3(e,2^L))\ge c_4\epsilon^4, \text{ for } 0\le j \le \frac{L}{10N}-1,\]
where $N= \lfloor \log \frac{1}{\epsilon}\rfloor$, and $\hat{\mathfrak{C}}_j$ is defined in the first paragraph of Section \ref{sec: concentration}. We then use Theorem \ref{thm: concentration} with $N$ as above, and $C_0= c_4\epsilon^4$ to find a constant $c_5>0$ such that for $L-L'\ge 40 \lfloor \log \frac{1}{\epsilon}\rfloor$, we have
\[\mathbf{P}(\cap_{j= L'}^L E_j^c\mid A_3(e,2^L)) \le 2^{-c_5 \epsilon^4 \frac{L-L'}{\log \frac{1}{\epsilon}}}\]
By possibly decreasing $c_5$ to handle $L'$ with $L'\ge L-40 \lfloor \log \frac{1}{\epsilon}\rfloor$, this implies \eqref{eq: lasagna_supremo}.
\end{proof}

\section{U-shaped regions}\label{sec: ushaped}
Let $\epsilon \in (0,1/2)$ and recall $\mathbf{e}_1=(1,0)$. On the event $E_k(\{0,\mathbf{e}_1\},\epsilon),$ in the box $[-3\cdot 2^k,3\cdot 2^k]^2$ (see Figure \ref{fig: ushaped}), the U-shaped region
\[U(k)=\left[ [-3\cdot 2^k, 3 \cdot 2^k] \times [-\frac{1}{3} \cdot 2^k, 3 \cdot 2^k]\right] \setminus (-\frac{7}{3} \cdot 2^k, \frac{7}{3} \cdot 2^k)^2,\]
contains an open arc on scale $2^k$, joining two five-arm points $\star_1\in B_1$ and $\star_2 \in B_2$. This arc is contained in the smaller region
\[\tilde{U}(k)\cup \tilde{V}(k) \subset U(k)\]
defined in \eqref{eqn: Uktilde-def}.

\begin{figure}
\centering
\scalebox{0.8}{\includegraphics[trim={0 0 0 0},clip]{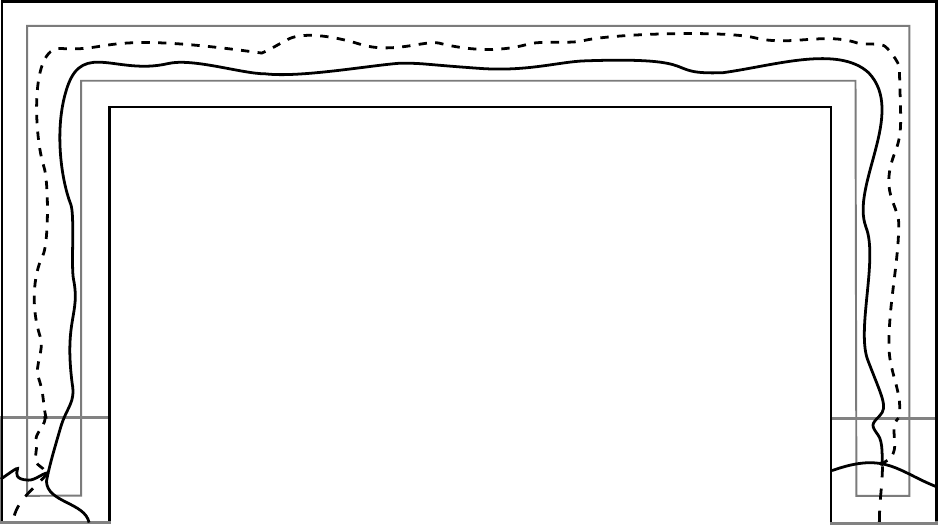}}
\caption{The U-shaped region. The region $\widetilde U \cup \tilde {V}$ appears in grey.}
\label{fig: ushaped}
\end{figure}

Recall that we denote by $\mathfrak{s}_k$ an arc in $U(k)$ connecting the two five-arm points with the minimal number of edges. On $E_k$, we also define $\ell_k$ to be the outermost open arc in $U(k)$ connecting $\star_1$ to $\star_2$, that is, the open arc $s$ in $U(k)$ whose initial and final edges are the vertical edges out of the five-arm points $\star_1$ and $\star_2$, and such that the compact region enclosed by the union of $s$ and the dual closed arc $\mathfrak{c}$ between $\star_1$ and $\star_2$ (item 8. in the definition of $E_k'$, in green in Figure \ref{fig: small-2}) is minimal. Note that since $r \subset \tilde{U}(k) \cup \tilde{V}(k)$, and $E_k'$ implies the existence of an open path connecting $\star_1$ and $\star_2$ inside $\tilde{U}(k)$, we have
\[\ell_k\subset \tilde{V}(k)\cup \tilde{U}(k).\]
In particular:
\begin{equation}\label{eqn: distance-lb}
\mathrm{dist}(\ell_k, \partial U(k))\ge \frac{1}{6}\cdot 2^k.
\end{equation}

The exact analogue of Proposition \ref{prop: implies-S} holds in $U(k)$ with $l_n$ replaced by $\ell_k$, the key point being that belonging to the outermost arc $\ell_k$ is characterized locally by a three-arm event. By comparison with $\#\ell_k$ we have (see \cite[Lemma 5.3]{DHSchemical1} for a similar estimate) for all $k\ge 1$ and $\epsilon \in (0,1/2)$,
\begin{equation}\label{eqn: initial}
\mathbf{E}[\# \mathfrak{s}_k \mid E_k']\le C2^{2k}\pi_3(2^k).
\end{equation}
By Proposition \ref{prop: e4-lwr}, this implies, for $\epsilon>0$
\begin{equation}\label{eqn: ej-lwr-bd}
\mathbf{P}(E_j(e,\epsilon,C)\mid A_3(e,2^j))\ge c_3\epsilon^4.
\end{equation}

To use \eqref{eqn: ej-lwr-bd} to construct a shorter path in the next section, we need the following:
\begin{prop}\label{prop: changecond}
There exists $C>0$ with the following property. For any $x_1\in B_1$, $x_2\in B_2$, $e\in \tilde{U}(k)$, $\epsilon \in (0,1/4)$, $d=2^j$, $j<k$ such that $B(e,d)\subset U(k)$, and $x_i\notin B(e,4d)$, $i=1,2$, and any event $E$ depending only on the status of edges in $B(e,d/100)$, we have
\begin{equation}\label{eqn: changecond}
\mathbf{P}(E\mid E_k', e\in \ell_k, \star_i=x_i, i=1,2)\le C\mathbf{P}(E \mid A_3(e,d))
\end{equation}
\end{prop}
\begin{proof}
This is a gluing argument very similar to \cite[Proposition 5.1]{DHSchemical1}. The main difference is the presence of five-arm points, and the closed dual path, but they do not add any essential difficulty. The case of $e\in B_1$ is illustrated in Figure \ref{fig: fivearms}. The remaining cases are similar or simpler.

\begin{figure}
\centering
\scalebox{0.8}{\includegraphics[trim={0 0 0 0},clip]{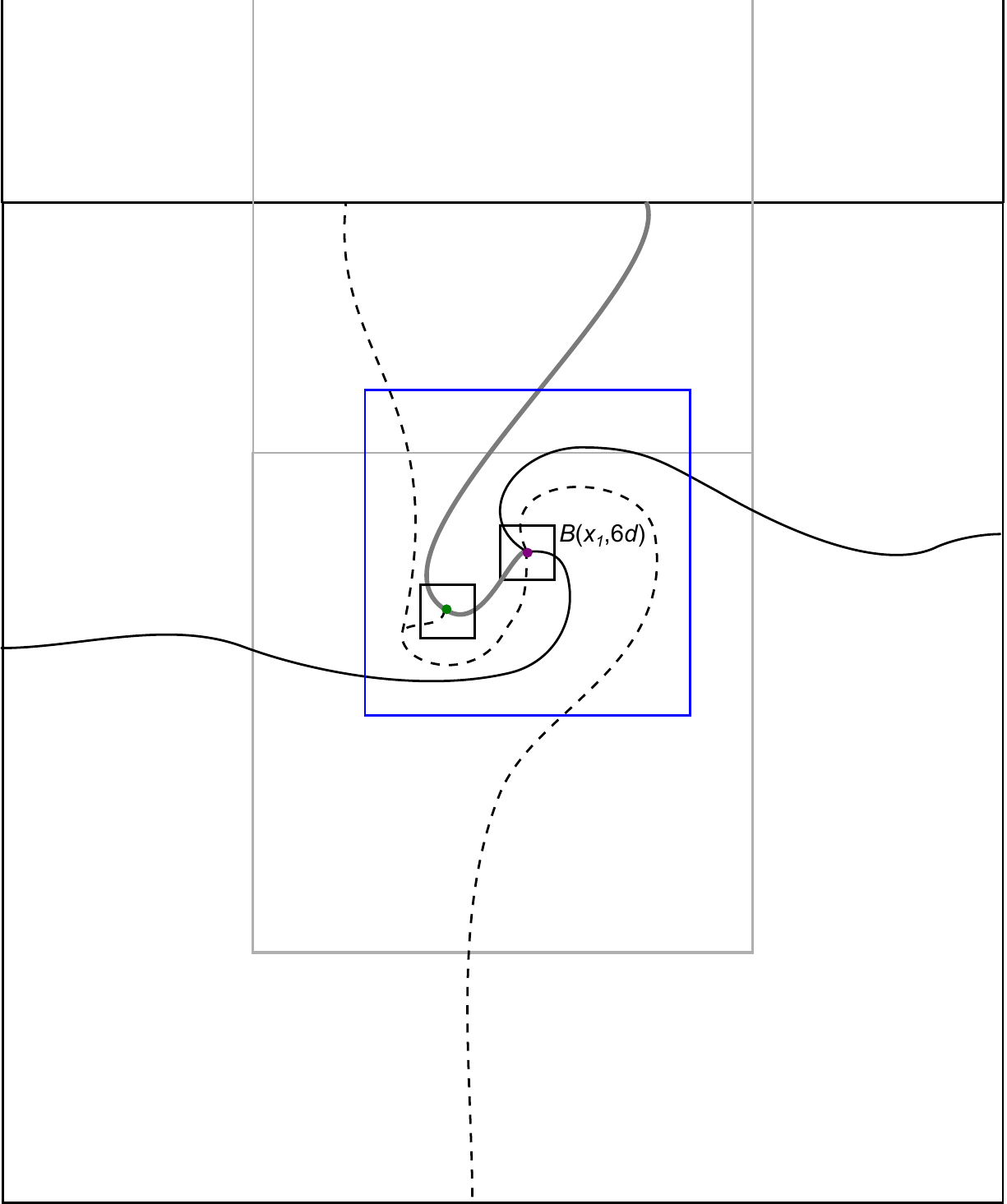}}
\caption{The gluing construction in $B_1'$, the with the same center as $B_1$ and twice the side length, to obtain \eqref{eqn: changecond}. The purple dot represents $\star_1$, and the green dot represents the edge $e$. The small boxes around $\star_1$ and $e$ have side length $d$, and $\mathrm{dist}(e,\star_1)\ge 4d$. The blue box is centered around $\star_1$ and has side length $6d$. The portion of $\ell_k$ inside $B_1$ appears as a thicker grey curve. The event $\{e\in \ell_k, \star_1 = x_1\}$  implies a) the existence of three arms from $e$ to the boundary of the small box centered at $e$; b) five arms from $x$ to the boundary of the small box centered at $x$, and c) five arms from the boundary of the blue box $B(x,6d)$ to $\partial B_1'$, with the landing areas as prescribed in the definition of $E_k'$. Conversely, given a), b) and c), with proper landing areas prescribed, we can make gluing constructions on scale $d$ to force the occurrence of  $\{e\in \ell_k, \star_1 = x\}$.}
\label{fig: fivearms}
\end{figure}
\end{proof}

\section{Iteration}\label{sec: iteration}
Our goal in this section is to derive the following proposition, which we use in Section~\ref{sec: main_proof} to prove the main result, Theorem \ref{thm: main}:
\begin{prop}\label{prop: main}
There exist constants $C, C'$ such that for any $\epsilon>0$ sufficiently small, $L\ge 1$, and $2^k\ge (C\epsilon^{-4} (\log\frac{1}{\epsilon})^2)^L$, we have
\begin{equation}
\mathbf{E}[\#\mathfrak{s}_k \mid E_k']\le (C'\epsilon^{1/2})^L 2^{2k}\pi_3(2^k).
\end{equation}
\end{prop}

The proof of Proposition~\ref{prop: main} is split into four sections. In Section~\ref{sec: construction}, we construct a family of candidate paths $(\sigma(i))_{i \geq 1} = (\sigma(i,k))_{i \geq 1}$ between the five arm points in $U(k)$ using lower-scale optimal paths and give the central iterative bound on their lengths in Proposition~\ref{prop: central}. In the remaining sections, we estimate the right side of this inequality: in Section~\ref{sec: some_definitions}, we present basic inequalities and choices of parameters, in Section~\ref{sec: case_one}, we give the bound in the case $i=0$, and in Section~\ref{sec: case_two}, we give the general case, $i \geq 1$.

\subsection{Construction and estimation of shorter arcs}\label{sec: construction}
Proposition \ref{prop: main} follows from an iterative procedure wherein improvements on the outermost arc $\ell_k$ in $U(k)$ (which is actually in the smaller region $\tilde{U}(k)\cup \tilde{V}(k)$) are made on larger and larger scales. The best improvement so far on scale $l$ is described by a sequence of parameters $\kappa_l(i)$, $l,i=1,2,\ldots$, nonincreasing in $l$, where $i$ denotes the number of the current iteration in the argument. All definitions in this section will depend on the number of iterations so far, which we will call the generation $i$.  The following is a key definition. It should be compared to Definition \ref{def: shortcuts}, where the shortcuts were constructed around the lowest crossing $l_n$ of the box $B(n)$. Here the shortcuts are constructed around $\ell_k$ in the region $U(k)$ (see Section \ref{sec: ushaped}).
\begin{df} \label{def: const} We say $r$ is a \underline{size $l$} shortcut in generation $i$ if
\begin{enumerate}
\item $r$ is an $\kappa_l(i)$-shortcut in the sense of Definition \ref{def: shortcuts}. In particular, the ``gain factor'' $\#r/\#\tau$ is $\le \kappa_l(i)$, where $\tau$ is the detoured part.
\item The shortcut $r$ is contained in a box of side length $3\cdot 2^l$. %(We say a box has size $s$ if it has side length $2s$.)
\item The detoured part $\tau$ is contained in a box $B\subset U(k)$, with the same center as the box in the previous item, of side length $2^{ \log \frac{1}{\epsilon} } 2^l$, has $\ell^\infty$-diameter greater than $\frac{2}{3}2^{\log \frac{1}{\epsilon}} 2^l$, and
\begin{equation}\label{eqn: dist}
    \mathrm{dist}(\tau,\{\star_1,\star_2\})\ge \frac{1}{8}2^{\log \frac{1}{\epsilon}}2^l.
\end{equation}
\end{enumerate}
\end{df}
Eventually, the gain factor will have the form $\kappa_l(i) = \epsilon^{C\min\{ i, C'l\}}$. We note that if $\epsilon$ is sufficiently small, the largest possible size of shortcut is no larger than $k+1$. Furthermore, distinct shortcuts (regardless of their sizes) are either nested or disjoint. By nested, we mean that the region enclosed by the union of a shortcut and its detoured section of $\ell_k$ surrounds that of  another shortcut. Both of these statements follow from the presence of ``shielding'' paths in item 4 of Definition \ref{def: shortcuts}. (See \cite[Prop. 2.3]{DHSchemical1}.) Last, the definition of size $l$ shortcuts is designed so that if $e\in \ell_k$ and if $E_l(e,\epsilon,\kappa_l(i)/\epsilon)$ occurs for an $l$ such that (a) $B(e,2^{l+\lfloor \log \frac{1}{\epsilon}\rfloor})\subset U(k)$ (which holds for $l\le k-3-\log \frac{1}{\epsilon}$ by \eqref{eqn: distance-lb}) and (b) $B(e,2^{l+\lfloor \log \frac{1}{\epsilon}\rfloor})$ does not contain the five-arm points $\star_i$, then there is a size $l$ shortcut in generation $i$ around $e$. This follows from the analogue of Proposition \ref{prop: implies-S} for U-shaped regions (which gives item 1 above) and the construction of events $E_k$ in the previous sections (the red box in Figure \ref{fig: main-Ek} for item 2 and the larger box from that figure and the existence of three-arm points in the rectangle $R$ in \eqref{eqn: R-def} for item 3.)

\paragraph{Construction.} Given the occurrence of $E_k'$, we define an arc $\sigma=\sigma(i)$ joining the two five-arm points in $U_k$ as follows. For each $l=k+1, k, \ldots, 1$  in order, choose a maximal collection of (generation $i$) shortcuts of size $l$, in the following way. First, we select a collection of size $k+1$ shortcuts such that no two of their detoured paths share vertices and the total length of the detoured sections of $\ell_k$ is maximal. The remaining uncovered portion of $\ell_k$ splits into a union of disjoint segments. For each such segment, we select a collection of size $k$ shortcuts such that no two of their detoured paths share vertices and the total length of the detoured sections of the segment is maximal. Continuing this way down to size 1 shortcuts, we obtain our maximal collection of shortcuts. Next we form the arc $\sigma$ consisting of the union of these shortcuts, and all the segments of $\ell_k$ which are not covered by this collection. It can be argued similarly to \cite[Lemma 2.4]{DHSchemical1} that what results from  the preceding construction is an open arc between the two five-arm points. Since the shortcuts are either nested or disjoint, this construction has the following essential property:

\begin{claim}\label{claim: comb}
Given any edge $e$ of the outermost arc $\ell_k$ of $U(k)$, if, after applying the above construction, the new arc $\sigma$ does not include a shortcut around $e$ of any size $l=k+1,k,\ldots,r-1$, then there is no shortcut of any size $k+1,k, \ldots, r-1$ around $e$ at all. 
\end{claim}
\begin{proof}
  Suppose $\sigma$ does not include a shortcut around $e$ of any size $l=k+1,k, \ldots, r-1$. Then for any such $l$, $e$ must be on a segment $\pi_l$ of $\ell_k$ that is uncovered after we place size $l$ shortcuts of $\ell_k$, and $\pi_l\subset \pi_{l+1}$ for all $l$, where we write $\pi_{k+2}=\ell_k$. If there is a  shortcut $r$ of size $l'$ (not contained in $\sigma$) around $e$ for some $l'=k+1,k,\ldots,r-1$, then note that $r$ must have both of its endpoints on $\pi_{l'+1}$. This is trivial if $\pi_{l'+1}=\ell_k$; otherwise, the segment $\pi_{l'+1}$ has endpoints which are starting vertices of shortcuts $r_1$, $r_2$ of sizes $\ge l'+1$. (If one endpoint of $\pi_{l'+1}$ is one of the five-arm points $\star_i$, we only get one such shortcut $r_1$.) Because shortcuts are nested, if $r$ has an endpoint on $\ell_k \setminus \pi_{l'+1}$, then the detoured path $\tau_i$ of some  $r_i$ would be contained in the detoured path $\tau$ of $r$. However, this is impossible by size considerations:
  \[ \frac{2}{3}2^{\log \frac{1}{\epsilon}}2^{ l'+1}\le \text{diam}~\tau_i \le\text{diam}~\tau \le 2^{\log \frac{1}{\epsilon}}2^{ l'}.\]
Therefore $r$ has both endpoints on $\pi_{l'+1}$. Because $\pi_{l'+1}$ is uncovered when we add size $l'$ shortcuts, and all such shortcuts are disjoint, maximality dictates that we must add $r$, or another shortcut of size $l'$ that covers $e$, to $\sigma$. This is a contradiction.
\end{proof}

From Claim \ref{claim: comb}, we see that if the new arc $\sigma$ contains a shortcut around $e$ of size $l$, then there is no shortcut of any size  $l+1, \ldots, k+1$ around $e$ at all. Indeed, $e$ must have been on an uncovered segment directly before we added shortcuts of size $l$, and is therefore not covered by a shortcut in $\sigma$ of any size $l+1, \ldots, k+1$.

The following proposition is the main iterative bound of the paper.
\begin{prop}\label{prop: central}
Let $\epsilon>0$ and fix $i\in \mathbb{N}$. Suppose moreover that, for some nonincreasing sequence of parameters $\delta_l(i)$, $l\ge 1$, we have
\begin{equation}\label{eqn: assumed-bds}
  \mathbf{E}[\#\mathfrak{s}_l \mid E_l'] \le \delta_l(i) 2^{2l}\pi_3(2^l).
\end{equation}
Let
\[\kappa_l(i):= \begin{cases}
\epsilon \cdot \delta_l(i) & \quad \text{if } l \geq 1 \\
1 & \quad \text{if } l \leq 0.
\end{cases}
\]
and $\sigma=\sigma(i)$ be defined as above, in terms of the sequence $\kappa_l(i)$, in the region $U_k$ for some $k\ge 1$. 
For $d=0, \ldots, k+1$, let $M>0$ and $d_1=d_1(d)$ be given as
\begin{equation}\label{eqn: d1-def}
d_1= d-M\epsilon^{-4}\left(\log\frac{1}{\epsilon}\right)^2.
\end{equation}
There are positive constants $c_*$ and $C_2$ with $C_2\ge 1$ such that for any $\epsilon$ sufficiently small, any $M>0$ and $i\in \mathbb{N}$, any parameters $\delta_l(i)$ as above, and any $k\ge 1$,
\begin{align}
\mathbf{E}[\#\mathfrak{s}_k \mid E_k'] & \le \mathbf{E}[\# \sigma(i) \mid E_k'] \nonumber \\
&\le C_2\sum_{d=0}^{k+1} 2^{2d}\pi_3(2^d)\cdot\big(2^{-\eta(\epsilon)d}+\sum_{s=1}^{d_1}2^{-\eta(\epsilon)(d-s)}\kappa_s(i)+\kappa_{d_1}(i)\big), \label{eqn: central}
\end{align}
where
\begin{equation}\label{eqn: eta-def}
\eta(\epsilon):= \frac{c_* \epsilon^4}{\log \frac{1}{\epsilon}}.
\end{equation}
\end{prop}
\begin{proof}
The first inequality follows because $\ell_k$ is in $\tilde U(k)\cup \tilde{V}(k) \subset U(k)$ and all its shortcuts are constructed in boxes in $U(k)$, so $\sigma$ remains in $U(k)$. To estimate the length of $\sigma$, we begin by dividing the outermost arc $\ell_k$, given $\sigma$, into a finite number of segments $\hat{\sigma}_\ell$, $\ell=1,\ldots$, where each segment $\hat{\sigma}_\ell$ is either
\begin{enumerate}
\item a single edge of the outermost arc also belonging to $\sigma$, or
\item a segment of the outermost arc which is detoured by a connected sub-segment of $\sigma$. That is, $\hat{\sigma}_\ell$ is the part of the outermost arc detoured by a shortcut $\sigma_\ell$ in $\sigma$.
\end{enumerate}
To each shortcut $\sigma_\ell$, we can associate a ``gain factor'' $\mathrm{gf}(\sigma_\ell)$, which is 1 if $\sigma_\ell$ is an edge of the outermost arc, and $\#\sigma_\ell /\# \hat{\sigma}_\ell$ otherwise.

By definition of $\sigma$, we have
\[\# \sigma = \sum_\ell \#\hat \sigma_\ell \times \mathrm{gf}(\sigma_\ell)\]
For a fixed generation $i$ (initially $i=1$), we organize this sum according to the size of the shortcut $\sigma_\ell$ (we say the size is 0 if there is no shortcut, in which case the gain factor is 1):
\[\# \sigma = \sum_{s=0}^{k+1}\sum_{\ell: \mathrm{size}(\sigma_\ell)=s} \#\hat \sigma_\ell \times \mathrm{gf}(\sigma_\ell).\]
Note that for large values of $s$, many of the summands will be zero because there cannot exist shortcuts of such sizes. Nevertheless, the bound holds as stated.

The event $E_k'$ is partitioned into the events:
\[F(x_1,x_2):=\{\star_1=x_1,\star_2=x_2\},  \quad x_1\in B_1, x_2\in B_2.\]
Note that $F(x,y)\cap F(x',y')=\emptyset$ unless $x=x'$ and $y=y'$. Thus, we have 
\[\#\sigma \le \sum_{x_1\in B_1,x_2\in B_2} \mathbf{1}_{F(x_1,x_2)}\sum_{s=0}^{k+1}\sum_{\ell: \mathrm{size}(\sigma_\ell)=s} \#\hat \sigma_\ell \cdot \kappa_s(i).\]
Next we divide the region $U(k)$ according to the distance $d$ to the points $x_1$, $x_2$, obtaining, for 
\[A_{d}=A_d(x_1,x_2) =\{e\in U(k): 2^{d}\le \mathrm{dist}(e,x_1)\le 2^{d+1} \text{ or } 2^{d}\le \mathrm{dist}(e,x_2)\le 2^{d+1} \},\]
(and $A_0=\{\mathrm{dist}(e,x_1)\le 1 \text{ or }  \mathrm{dist}(e,x_2) \le 1\}$) the decomposition
\[\#\sigma \le \sum_{x_1\in B_1,x_2\in B_2} \mathbf{1}_{F(x_1,x_2)}\sum_{d=0}^{k+1}\sum_{s=0}^{d_0} \sum_{\ell: \mathrm{size}(\sigma_\ell)=s} \#(\hat{\sigma}_\ell\cap A_{d} )\cdot \kappa_s(i).\]
Here $d_0= d_0(d)=\max(d+4- \log\frac{1}{\epsilon},0)$. We do not need to consider larger sizes since they cannot occur at such distances by the condition \eqref{eqn: dist}.

 By the remark following Claim \ref{claim: comb}, if a shortcut $\sigma_\ell$ surrounds an edge $e$ and has size $s<k+1$, then there is no shortcut of any size $l=k+1, \ldots, s+1$ around $e$ at all, so
\begin{equation}\label{eqn: stelmo}
\#\sigma \le \sum_{x_1\in B_1, x_2\in B_2} \mathbf{1}_{F(x_1,x_2)}\sum_{d=0}^{k+1}\left(\sum_{s=0}^{d_1} \#(B_s\cap A_d ) \cdot \kappa_s(i) + \#(\ell_k \cap  A_d) \cdot \kappa_{d_1}(i) \right),
\end{equation}
where $B_s=B_s(\kappa_s(i))$ is the set edges on $\ell_k$ with no generation $i$ shortcuts of sizes $l=k+1,k,\ldots, s+1$. We have used monotonicity of $\delta_{\ell}(i)$ in $\ell$. (Recall that $\kappa_{d_1}=1$ for $d \le M\epsilon^{-4}\left(\log \frac{1}{\epsilon}\right)^2$). 

From Propositions \ref{prop: concentration} (for which we use the assumed bounds \eqref{eqn: assumed-bds}) and \ref{prop: changecond}, and the fact that events $E_l(e)$ for $l$ such that the box $B(e,2^{l+\lfloor \log \frac{1}{\epsilon}\rfloor})\subset U(k)$ does not contain the five-arm points $\star_i$ guarantee the existence of size $l$ shortcuts (see the discussion below Definition \ref{def: const}), we have
\begin{equation}\label{eqn: napoleon}
\begin{split}
  \mathbf{P}( e\in B_s \mid E_k',  e\in \ell_k, F(x_1,x_2) )&\le  \mathbf{P}(\cap_{l=s+1}^{d-\log\frac{1}{\epsilon}-10}E_l(e,\epsilon, \delta_l(i))^c\mid E_k', e\in \ell_k, F(x_1,x_2))\\
  & \le C \mathbf{P}(\cap_{l=s+1}^{d-\log\frac{1}{\epsilon}-10}E_l(e,\epsilon, \delta_l(i))^c\mid A_3(e,2^d))\\
&\le C 2^{-\frac{c_* \epsilon^4}{\log \frac{1}{\epsilon}}(d-s)}.
\end{split}
\end{equation}
whenever $e\in A_d$. 
From \eqref{eqn: napoleon} and \eqref{eqn: stelmo}, we have the following estimate for the size of $\sigma$:
\begin{equation}\label{eqn: sigma-main}
\begin{split}
&\mathbf{E}[\# \sigma(i) \mid E_k']\\
\le &\sum_{d=0}^{k+1} \sum_{x_1\in B_1, x_2\in B_2} \mathbf{P}(F(x_1,x_2)\mid E_k') \\
&\quad \times\big[ \sum_{s=0}^{d_1}\kappa_s(i) \sum_{e\in A_d}\mathbf{P}(e\in B_s\mid E_k', F(x_1,x_2),e\in \ell_k)\mathbf{P}(e\in \ell_k\mid E_k',F(x_1,x_2))\\
&\qquad +\kappa_{d_1}(i)\sum_{e\in A_d} \mathbf{P}(e\in \ell_k \mid E_k', F(x_1,x_2))\big]\\
\le & C_2 \sum_{d=0}^{k+1}2^{2d}\pi_3(2^d)(2^{-\eta(\epsilon)d}+\sum_{s=1}^{d_1}2^{-\eta(\epsilon)(d-s)}\kappa_s(i)+\kappa_{d_1}(i)).
\end{split}
\end{equation}
In passing to the final line of \eqref{eqn: sigma-main}, we have used the estimate
\[\mathbf{P}(e\in \ell_k \mid E_k', F(x_1,x_2))\le C\pi_3(2^d),\]
for $e\in A_d$, where $C$ is some constant independent of the parameters (in particular, of the $x_i$'s). This is the analogue (for $\ell_k$ instead of the lowest crossing $l_n$) of the upper bound in estimate \eqref{eqn: 3-arm-bdy}. That the conditioning on $E_k'$ and $F(x_1,x_2)$ results only in an additional constant factor is shown by a gluing construction very similar to the one illustrated in Figure \ref{fig: fivearms}.
\end{proof}

\subsection{Some definitions}\label{sec: some_definitions}
In estimating the volume of the new path $\sigma(i)$, $i\ge 1$, using \eqref{eqn: sigma-main}, it is important to track the dependence on $\epsilon$ when performing the requisite summations. We begin by introducing some notations and simple bounds we will use repeatedly in Sections~7.3 and 7.4.

We first take $\epsilon>0$ sufficiently small that Proposition \ref{prop: central} holds. We will need $\epsilon$ to be possibly even smaller, and will state this at various points in what follows. A key point is that the size of $\epsilon$ always depends on fixed parameters, and never on $k$ or the generation $i$.

We define:
\[m= M\epsilon^{-4}(\log \frac{1}{\epsilon})^2,\]
with $M$ as in \eqref{eqn: d1-def}. To simplify notation, we will assume $\epsilon$, $M$ are taken so that $m$ is an integer. With this notation we have $d_1=d-m$.
Note also that
\[\eta(\epsilon) m = c_* M \log \frac{1}{\epsilon}.\]
For $l\ge1$, set
\begin{equation}\label{eqn: sl-def}
s_l=3ml=3Ml\cdot \epsilon^{-4}\left(\log\frac{1}{\epsilon}\right)^2, \quad l\ge 1,
\end{equation}
with $s_0=0$.

We define
\begin{align*}
\theta(\epsilon)&=\frac{2^{\eta(\epsilon)}}{2^{\eta(\epsilon)}-1}\le \frac{2}{2^{\eta(\epsilon)}-1},
\end{align*}
where the inequality holds if $\epsilon$ is sufficiently small. We choose $M$ such that
\begin{equation}
M>  \max(1,7/c_*). \label{eqn: M-cond}
\end{equation}
Since 
\begin{equation}\label{eqn: frog}
2^{\frac{c_*}{\log \frac{1}{\epsilon}}\epsilon^4 }-1\ge c_*\frac{ \ln 2}{\log \frac{1}{\epsilon}}\epsilon^4,
\end{equation}
we have
\begin{equation}\label{eqn: toledo}
\frac{\epsilon^{M c_*}}{2^{\frac{c_*\epsilon^4}{\log \frac{1}{\epsilon}}}-1} \le \frac{\log\frac{1}{\epsilon}}{c_* \ln 2} \epsilon^3.
\end{equation}
We will always choose $\epsilon=\epsilon(c_*)$ so small that the quantity in \eqref{eqn: toledo} is less than $2\epsilon^2$:
\begin{equation}\label{eqn: chicago}
\theta(\epsilon) \epsilon^{c_*M}\le \epsilon^2.
\end{equation}

The constant $C_3\ge 1$ is chosen such that for all $L\ge 1$ and any $\alpha\le 1$,
\begin{equation}\label{eqn: C3-def}
\begin{split}
\sum_{d=0}^{L+1} 2^{2d}\pi_3(2^d) 2^{-\alpha d}&= \pi_3(2^L) \sum_{d=0}^{L+1} \frac{\pi_3(2^d)}{\pi_3(2^L)}2^{(2-\alpha)d}\\
%&\le C\pi_3(2^L) \sum_{d=1}^{L+1} 2^{(2-\alpha)d}2^{\beta(L-d)}\\
&\le C_5\pi_3(2^L) 2^{\beta L} \frac{2^{(L+2)(2-\beta-\alpha)}-1}{2^{2-\beta-\alpha}-1}\\
&\le C_3 2^{(2-\alpha)L}\pi_3(2^L)
\end{split}
\end{equation}
Here $0<\beta<1$ was introduced in \eqref{eqn: leslie}.

%Finally, we will choose $\epsilon>0$ such that
%\begin{equation}\label{eqn: aizenman}
%\epsilon < \min\{\frac{1}{4\cdot 10^4}, \left(\frac{c_0^2}{2C_5 2^{2-\beta}}\right)^{1/\gamma}, \frac{1}{C_1^2}, 2^{-1/\gamma}\},
%\end{equation}
%(see \eqref{eqn: epsilon-const-1}).

\subsection{Improvement by iteration}\label{sec: case_one}
%Assuming \eqref{eqn: mountain}, we find by \eqref{eqn: sigma-main}.
%\begin{align}
%\mathbf{E}[\#\sigma(2) \mid E_k] &\le C_2\sum_{d=0}^{k+1} 2^{2d}\pi_3(2^d)\cdot\big(2^{-\eta(\epsilon)d}+ 2^{-\eta(\epsilon)d} \sum_{s=1}^{2m} 2^{\eta(\epsilon)s}+2\epsilon^2 \sum_{s=2m}^{d_1}2^{-\eta(\epsilon)(d-s)}\big) \label{eqn: whatnow}\\
%&+ C_2 \sum_{d=1}^{2m}2^{2d}\pi_3(2^d) +  C_2\sum_{d> 2m}^{k+1} 2^{2d}\pi_3(2^d)\kappa_{d_1}(2). \label{eqn: whynot}
%\end{align}
%Sums up to $d_1$ are taken to be 0 when $d\le m$. 

%The first term in \eqref{eqn: whynot} is
%\[C_2\sum_{d=0}^{2m} 2^{2d}\pi_3(2^d)\le C_2C_3 2^{2\cdot 2m}\pi_3(2^{2m}).\]
%Using \eqref{eqn: leslie}, this is smaller than $C_2C_3\epsilon^{3/2} 2^{2k}\pi_2(2^k)$ if $k \ge 3m$.

%The second sum in \eqref{eqn: whynot} is
%\begin{equation}\label{eqn: knight}
%2\epsilon^2\cdot \sum_{d= 2M\epsilon^{-4}(\log \frac{1}{\epsilon})^2+1}^{k+1}2^{2d}\pi_3(2^d),
%\end{equation}
%which is bounded by
%\begin{equation}\label{eqn: patriot}
%2\epsilon^2 C_32^{2k}\pi_3(2^k)
%\end{equation}
%by definition of $C_3$.
%Using \eqref{eqn: leslie}, the first and second terms in \eqref{eqn: patriot} are bounded by the second if
%\[2\cdot 2^{2(m-k)}\frac{\pi_3(2^m)}{\pi_3(2^k)} \le 2\epsilon.\]
%By \eqref{eqn: leslie}, we have
%\[2^{2(m-k)}\frac{\pi_3(2^m)}{\pi_3(2^k)}\le C_52^{(2-\beta)(m-k)},\]
%so if
%\begin{equation}\label{eqn: dropit}
%k\ge 2M \epsilon^{-4}(\log \frac{1}{\epsilon})^2\ge M\epsilon^{-4}(\log \frac{1}{\epsilon})^2+\frac{1}{2-\beta}\log C_5+\frac{1}{2-\beta}\log \frac{1}{\epsilon}.
% \end{equation}
We start from the initial estimate
\begin{equation}\label{eqn: initial}
\mathbf{E}[\#\mathfrak{s}_k \mid E_k']\le C_1 2^{2k}\pi_3(2^k),
\end{equation}
for some $C_1\ge 1$.
We apply Proposition \ref{prop: central} with $\delta_s(0)= C_1$ (equivalently, $\kappa_s(0)= C_1 \epsilon$) for all $s\ge 1$.
Defining the corresponding arc $\sigma(0)$, we obtain for $k\ge 1$,
\begin{equation}\label{eqn: sigma-1}
\mathbf{E}[\# \sigma(0)\mid E_k'] \le C_2 \sum_{d=0}^{k+1}2^{2d}\pi_3(2^d)(2^{-\eta(\epsilon)d}+C_1 \epsilon \sum_{s=1}^{d_1}2^{-\eta(\epsilon)(d-s)}+\kappa_{d_1}(0)).
\end{equation}
We use this last expression to obtain an improve on \eqref{eqn: initial} under the assumption
\begin{equation}\label{eqn: brief}
k>s_1=3m.
\end{equation}
The quantity \eqref{eqn: sigma-1} is bounded by
\begin{align}
 &C_2\sum_{d=0}^{k+1} 2^{2d}\pi_3(2^d)2^{-\eta(\epsilon)d} + C_2 C_1 \epsilon \sum_{d=0}^{k+1} 2^{2d}\pi_3(2^d)2^{\eta(\epsilon)} \frac{\epsilon^{Mc_*}}{2^{\eta(\epsilon)}-1} \label{eqn: incredible}\\
+& C_2 \sum_{d=0}^{k+1} 2^{2d}\pi_3(2^d)  \kappa_{d_1}(0) . \label{eqn: intrepid}
\end{align}

By definition of $C_3$ (see \eqref{eqn: C3-def}), the first term in \eqref{eqn: incredible} is bounded as follows
\begin{equation}
C_2\sum_{d=0}^{k+1} 2^{2d}\pi_3(2^d)2^{-\eta(\epsilon)d} \le C_2 C_3 2^{2k}2^{-\eta(\epsilon)k}\pi_3(2^k).
\end{equation}
Using \eqref{eqn: chicago} and \eqref{eqn: C3-def}, the second term in \eqref{eqn: incredible} is bounded by
\begin{equation}
2C_1C_2C_3 \epsilon^{3} 2^{2k}\pi_3(2^k).
\end{equation}
Similarly, for \eqref{eqn: intrepid} we have the upper bound
\begin{equation}
\begin{split}
    C_2\sum_{d=0}^m 2^{2d}\pi_3(2^d)+C_2\sum_{d=m+1}^{k+1}C_1\epsilon 2^{2d}\pi_3(2^k) &\le  C_2C_32^{2k}\pi_3(2^k)\big[ 2^{-(k-m)(2-\beta)}+C_1\epsilon\big]\\
    &\le 2C_1C_2C_3\epsilon 2^{2k}\pi_3(2^k).
  \end{split}
\end{equation}
In the second inequality we have assumed that $k\ge s_1=3m$ and taken $\epsilon$ sufficiently small (depending only on $\beta$).

Adding these three bounds, $\mathbf{E}[\# \sigma(0)\mid E_k']$ is bounded by
\begin{equation}\label{eqn: karldergrosse}
C_2C_32^{2k}\pi_3(2^k)(2^{-\eta(\epsilon) k}+2C_1\epsilon^3+2C_1\epsilon).
\end{equation}
\eqref{eqn: karldergrosse} is bounded by
\[2C_2C_3 \epsilon^{1/2}2^{2k}\pi_3(2^k).\]
if $\epsilon$ is small (depending on $C_1$) and \eqref{eqn: brief} holds,
since then $2^{-\eta(\epsilon)k}\le \epsilon^{1/2}$.
Thus,
\[\mathbf{E}[\#\sigma(0) \mid E_k'] \le 4C_2C_3\epsilon^{1/2} \cdot  2^{2k}\pi_3(2^k)\]
for $k$ satisfying \eqref{eqn: brief}.

This completes our bounding the right side of the main inequality in Proposition \ref{prop: central}. In summary, we now have $\mathbf{E}[\# \mathfrak{s}_r \mid E_r'] \le \delta_r(1)2^{2r}\pi_3(2^r)$, with
\begin{equation}
\label{eqn: getoor}
\delta_r(1)=\begin{cases}
  C_1 ,&\quad r \le s_1,\\
  4C_2C_3 \epsilon^{1/2},&\quad r> s_1.
  \end{cases}
\end{equation}
We may now iterate Proposition \ref{prop: central} for further generations to obtain an improved bound. We formulate a general inductive result in the next section.

\subsection{General case}\label{sec: case_two}
\begin{prop} \label{prop: general-iteration}
  Assume that
\begin{equation}\label{eqn: glueck}
\mathbf{E}[\# \mathfrak{s}_r \mid E_r'] \le \delta_r(L)2^{2r}\pi_3(2^r),
\end{equation}
  holds for the choice of parameters
\begin{equation}
\delta_r(L)=\begin{cases}
C_1 & \quad \text{ if }  r \le s_1 \\
 (4C_2C_3)^l\epsilon^{l/2} & \quad \text{ if }   s_l < r \le  s_{l+1}, \  l=1,\ldots,L-1,\\
 (4C_2C_3)^L\epsilon^{L/2}& \quad \text{ if } r> s_L.
\end{cases}
\end{equation}
Then, \eqref{eqn: glueck} also holds for $r\ge s_{L+1}$ and $\delta_r(L)$ replaced by
\begin{equation}\label{eqn: hans}
\delta_r(L+1) =(4C_2C_3)^{L+1}\epsilon^{(L+1)/2}.
\end{equation}
\end{prop}
\begin{proof}
By \eqref{eqn: getoor}, we may assume $L\ge 1$ and $r=k \ge s_{L+1}=3m(L+1)$. Start from an upper bound for the main inequality of Proposition \ref{prop: central}:
\begin{align}
\mathbf{E}[\#\sigma(L) \mid E_k'] \le&~C_2 \sum_{d=0}^{k+1} 2^{2d}\pi_3(2^d)2^{-\eta(\epsilon)d} \times\\
&\  \big(\sum_{s=0}^{3m}2^{\eta(\epsilon)s}+\epsilon \sum_{l=1}^L \sum_{s= (s_{l-1}+1)\wedge d_1}^{s_l \wedge d_1 }2^{\eta(\epsilon)s} (4C_2C_3\epsilon^{1/2})^{l-1}+ \epsilon(4C_2C_3\epsilon^{1/2})^L\sum_{s>s_L}^{d_1}2^{\eta(\epsilon)s}\big) \label{eqn: matthew}\\
+& C_2\sum_{d\le m}2^{2d}\pi_3(2^d) +C_2 \sum_{d>m}^{k+1} 2^{2d}\pi_3(2^d)\cdot  \kappa_{d_1}(L). \label{eqn: jonas}
\end{align}
The final sum in \eqref{eqn: matthew} is zero if $s_L \ge d_1$. The term \eqref{eqn: jonas} corresponds to the $\kappa_{d_1(i)}$ term in \eqref{eqn: central}. The term \eqref{eqn: matthew} corresponds to $2^{-\eta(\epsilon)d}$ plus the term over sizes $1\le s < d_1$ in \eqref{eqn: central}. Sizes $0\le s\le s_1$ are bounded by the first term. Other sizes are split over ranges of $(s_{l-1},s_l]$ up to $d_1$ in the second term of \eqref{eqn: matthew} and sizes $0\le s\le s_1$ are double counted from the previous term.

\subsubsection{The $\kappa_{d_1}$ term \eqref{eqn: jonas}}
The term \eqref{eqn: jonas} is bounded (since $C_1\ge 1$) by
\begin{align}
%&\sum_{d\le 2m} 2^{2d}\pi_3(2^d)+2\epsilon^2 \sum_{l=1}^L \sum_{s_{l-1} < d_1 \le s_{l}} (4C_2C_3\epsilon^{1/2})^{l-1} 2^{2d}\pi_3(2^d)\\
%& +\epsilon^2 (4C_2C_3)^L \epsilon^{L/2}\sum_{d_1> s_L}^{k+1}2^{2d}\pi_3(2^d) \nonumber \\
%\le 
& C_2 \sum_{d\le m} 2^{2d}\pi_3(2^d)+C_1C_2 \epsilon \sum_{l=1}^L(4C_2C_3\epsilon^{1/2})^{l-1}\sum_{d=3m(l-1)+m+1}^{3ml+m}2^{2d}\pi_3(2^d) \label{eqn: zachariah} \\
+ & \epsilon C_2(4C_2C_3)^L \epsilon^{L/2}\sum_{d:d_1> s_L}^{k+1}2^{2d}\pi_3(2^d). \nonumber
\end{align}
By \eqref{eqn: C3-def}, the second sum in \eqref{eqn: zachariah} is bounded by $C_1C_2 \epsilon$ times
%\begin{align}
%&\sum_{l=1}^L(4C_2C_3\epsilon^{1/2})^{l-1}\sum_{d=3m(l-1)+m+1}^{3ml+m}2^{2d}\pi_3(2^d)\\
%\le & 
\begin{equation}
\sum_{l=1}^L(4C_2C_3\epsilon^{1/2})^{l-1} C_3 2^{2(3ml+m)}\pi_3(2^{3ml+m}).\label{eqn: barabas}
\end{equation}
%\end{align}
Using \eqref{eqn: leslie}, \eqref{eqn: barabas} is no greater than
\begin{equation}
%&\sum_{l=1}^L(4C_2C_3\epsilon^{1/2})^{l-1} C_3 2^{2(3ml+m)}\pi_3(2^{3ml+m}) \nonumber \\
%\le & C_5C_3\pi_3(2^{3mL+m})2^{3\beta  mL+2m}\sum_{l=1}^L (4C_2C_3\epsilon^{1/2})^{l-1}2^{(2-\beta)3ml} \nonumber \\
C_5 C_3  (4C_2C_3\epsilon^{1/2})^{-1}2^{3\beta m L+2m} \pi_3(2^{3mL+m}) \sum_{l=1}^L 2^{(2-\beta)3ml} 2^{l\log 4C_2C_3\epsilon^{1/2}} \label{eqn: sum-0}.\\
%\le & C_5C_3 (4C_2C_3\epsilon^{1/2})^{-1} 2^{2\cdot (3m L+m)}\pi_3(2^{3mL+m}) (4C_2C_3\epsilon^{1/2})^L \frac{4C_2C_3 \epsilon^{1/2}\cdot 2^{(2-\beta)3m}}{4C_2C_3 \epsilon^{1/2} \cdot 2^{(2-\beta)3m}-1} \nonumber\\
%\le & 2\cdot C_5C_3 2^{2\cdot (3mL+m)}\pi_3(2^{3mL+m}) (4C_2C_3\epsilon^{1/2})^{L-1} \nonumber\\
\end{equation}
The sum in \eqref{eqn: sum-0} is bounded by $2^{(2-\beta)3mL}(4C_2C_3 \epsilon^{1/2})^L$ times
\[\frac{4C_2C_3 \epsilon^{1/2}\cdot 2^{(2-\beta)3m}}{4C_2C_3 \epsilon^{1/2} \cdot 2^{(2-\beta)3m}-1}\le 2,\]
if $m\ge \frac{1}{6}\frac{1}{2-\beta}\log \frac{1}{\epsilon}$. This is true for $\epsilon$ small enough (depending on $\beta$). Thus using \eqref{eqn: leslie}, \eqref{eqn: sum-0} is bounded by
\begin{equation}2\cdot C_5^2 C_3 2^{2(\beta-2)m}2^{2\cdot 3m(L+1)}\pi_3(2^{3m(L+1)}) (4C_2C_3\epsilon^{1/2})^{L-1}. \label{eqn: parent}
\end{equation}
%\begin{align*}
%2^{2\cdot (3mL+m)}\pi_3(2^{3mL+m})&= 2^{2\cdot (3mL+3m)} 2^{-4m} \pi_3(2^{3mL+3m}) \frac{\pi_3(2^{3mL+m})}{\pi_3(2^{3mL+3m})}\\
%&\le 2^{-4m} 2^{2\cdot 3m(L+1)} C_5 2^{2m \beta} \pi_3(2^{3mL+3m})\\
%&\le C_5 2^{2(\beta-2)m} 2^{2\cdot 3m(L+1)}\pi_3(2^{3m(L+1)}).
%\end{align*}

%\[4C_2C_3 \epsilon^{1/2}\cdot 2^{(2-\beta)3m} \ge 2.\]
%This in turn will be true if

Recalling the extra factors $C_1 C_2$ and $\epsilon$, we find from  \eqref{eqn: zachariah} and \eqref{eqn: parent} that \eqref{eqn: jonas} is bounded by
\begin{equation}\label{eqn: darwin}
\begin{split}
&C_2C_3 2^{2 m}\pi_3(2^{m})\\
+&~2C_1C_5^2 (4C_2C_3\epsilon^{1/2})^{L} \epsilon^{1/2}2^{2(\beta-2)m}2^{2\cdot 3m(L+1)}\pi_3(2^{3m(L+1)}) \\
+&~\epsilon C_2C_3(4C_2C_3)^L \epsilon^{L/2} 2^{2k}\pi_3(2^k).
\end{split}
\end{equation}
We compare the first two terms in \eqref{eqn: darwin} to the third using \eqref{eqn: leslie}. We have:
\begin{align}
2^{2k}\pi_3(2^k) 2^{2\cdot (m-k)}\frac{\pi_3(2^{m})}{\pi_3(2^{k})} &\le C_5 2^{2k}\pi_3(2^k) 2^{(2-\beta)(m-k)}\\
&\le C_5 2^{2k}\pi_3(2^k) 2^{m-k}. \label{eqn: ketty}
\end{align}
In the second step we have used $\beta<1$ and $k\ge s_L+m$. Then if $\epsilon$ is small enough (depending on $C_5$), \eqref{eqn: ketty} is bounded by
\begin{equation}\label{eqn: tory}
C_5 2^{2k}\pi_3(2^k) 2^{-3mL}\le C_5 2^{2k}\pi_3(2^k)\epsilon^{10L} \le 2^{2k}\pi_3(2^k)\epsilon^{9L}.
\end{equation}

For the second term in \eqref{eqn: darwin}, we find (using $\beta<1$ and $k\ge s_{L+1}$) for $\epsilon$ small (depending on $\beta$):
\begin{equation}\label{eqn: fritz}
 2^{2(\beta-2)m}2^{2\cdot 3m(L+1)}\pi_3(2^{3m(L+1)})\le C_5\epsilon^{10} 2^{2k}\pi_3(2^k).
%&\le C_52^{-m}2^{2k}\pi_3(2^k)\\
\end{equation}
Putting \eqref{eqn: tory} and \eqref{eqn: fritz} into \eqref{eqn: darwin}, we find that \eqref{eqn: jonas} is bounded by
\begin{equation}\label{eqn: jeremiah}
\epsilon(4C_2C_3\epsilon^{1/2})^L2^{2k}\pi_3(2^k)(C_2C_3+2C_1C^3_5 \epsilon^{19/2}) + C_2C_3\epsilon^{9L}  2^{2k}\pi_3(2^k),
\end{equation}  
when $k\ge s_{L+1}= s_L+3m$.

\subsubsection{Term \eqref{eqn: matthew}: case $s_L \le d_1$}
For \eqref{eqn: matthew}, we distinguish the cases when $s_L\le d_1$ and $s_L> d_1$. In the first case, the term in question is,
\begin{equation}\label{eqn: rachel}
\sum_{s=0}^{3m}2^{\eta(\epsilon)s}+\epsilon \sum_{l=1}^L \sum_{s= s_{l-1}+1}^{s_l }2^{\eta(\epsilon)s} (4C_2C_3\epsilon^{1/2})^{l-1}+\epsilon (4C_2C_3\epsilon^{1/2})^L\sum_{s>s_L}^{d_1}2^{\eta(\epsilon)s}.
\end{equation}

By a summation like the one leading to \eqref{eqn: parent}, the middle term in \eqref{eqn: rachel} is bounded by
%\begin{align}
%&2\epsilon^2\sum_{l=1}^L(4C_2C_3\epsilon^{1/2})^{l-1}\sum_{s=3m(l-1)+1}^{3ml} 2^{\eta(\epsilon)s} \nonumber \\
% \le&2\epsilon^2\sum_{l=1}^L(4C_2C_3\epsilon)^{l-1}\theta(\epsilon) 2^{c3Ml\log \frac{1}{\epsilon}} \nonumber \\
%=& 2(4C_2C_3\epsilon^{1/2})^{-1}\epsilon^2\theta(\epsilon) \sum_{l=1}^L 2^{l (3c_*M\log \frac{1}{\epsilon}+\log 4C_2C_3\epsilon^{1/2})} \nonumber \\
%\le& 2(4C_2C_3\epsilon^{1/2})^{-1} \epsilon^2 \theta(\epsilon) \frac{2^{(3c_*M\log \frac{1}{\epsilon}+\log 4C_2C_3\epsilon^{1/2})}}{2^{(3c_*M\log \frac{1}{\epsilon}+\log 4C_2C_3\epsilon^{1/2})}-1}  2^{L (3c_*M \log \frac{1}{\epsilon}+\log 4C_2C_3\epsilon^{1/2})} \label{eqn: joey}
%\end{align}
%If $s_L\le d_1$, then $d\ge 3mL+m$, so $2^{-\eta(\epsilon)d}$ times the quantity in \eqref{eqn: joey} is bounded by
\begin{equation}\label{eqn: charlene}
 4\epsilon\cdot (4C_2C_3\epsilon^{1/2})^{-1} \theta(\epsilon) 2^{\eta(\epsilon)d_1}(4C_2C_3\epsilon^{1/2})^L\le 2^{\eta(\epsilon) d} \epsilon^{2} (4C_2C_3\epsilon^{1/2})^{L} .
\end{equation}
%We have used \eqref{eqn: chicago}, and 
%\begin{equation}\label{eqn: vegas}
%\frac{2^{(3c_*M\log \frac{1}{\epsilon}+\log 4C_2C_3\epsilon^{1/2})}}{2^{(3c_*M\log \frac{1}{\epsilon}+\log 4C_2C_3\epsilon^{1/2})}-1}\le 2,
%\end{equation}
%which holds if
%\[3c_*M\log\frac{1}{\epsilon}+\log 4C_2C_3\epsilon^{1/2}\ge \log 2,\]
%or, using \eqref{eqn: M-cond},
%\[20\log \frac{1}{\epsilon}\ge \log 2.\]
%This is certainly true for $\epsilon \le 1/2$.

The first and third terms in \eqref{eqn: rachel} are bounded, respectively, by 
\[\theta(\epsilon) 2^{3m\eta(\epsilon)}\]
and
\[(4C_2C_3\epsilon^{1/2})^L\epsilon \theta(\epsilon) 2^{\eta(\epsilon)d_1}.\]
Multiplying these bounds by $2^{-\eta(\epsilon)d}$, and using \eqref{eqn: M-cond} and \eqref{eqn: chicago} we find an estimate of
\begin{equation}\label{eqn: chirac}
\epsilon \cdot \epsilon^{L} + \epsilon^2(4C_2C_3\epsilon^{1/2})^L
\end{equation}
if $L\ge 1$ and $d_1\ge s_L$. Here we have taken $\epsilon$ small depending on $C_2$ and $C_3$.

%Multiplying this bound by $2^{-\eta(\epsilon)d}$, we have,
%\begin{equation}
%\theta(\epsilon) 2^{-\eta(\epsilon)(d-2m)}.
%\end{equation}
%for $d\ge 3mL+m$, this is bounded, using \eqref{eqn: M-cond} and \eqref{eqn: chicago} by 
%\begin{equation}\label{eqn: chirac}
%le 2\epsilon^2 \cdot \epsilon^{7(3L-2)}
%\theta(\epsilon) \epsilon^{c_*M(3L-1)}\le 2\epsilon^2 \cdot \epsilon^{7L},\end{equation}
%if $L\ge 1$. 

%For the third term in \eqref{eqn: rachel}, we have a bound of
%Multiplied by $2^{-\eta(\epsilon)d}$, this is bounded by:
%\begin{equation}\label{eqn: morris}
%2\epsilon^4(4C_2C_3\epsilon^{1/2})^L
%\end{equation}
%using \eqref{eqn: chicago}.

Using \eqref{eqn: chirac}, \eqref{eqn: charlene} and performing the sum over $d$, we find that the contribution to \eqref{eqn: matthew} from $d_1\ge s_L$ is
\begin{equation}\label{eqn: rembrandt}
C_2C_3\epsilon^2 (4C_2C_3\epsilon^{1/2})^{L} 2^{2k}\pi_3(2^k) + \epsilon C_2C_3 \epsilon^{L} 2^{2k}\pi_3(2^k).
\end{equation}

\subsubsection{Term \eqref{eqn: matthew}: case $s_L>d_1$.}
We turn to the case $s_L> d_1$. We let
\begin{align*}
l_d &= \max\{l: s_{l}\le d_1\}\\
%&= \max\{l: 3ml \le d-m\}\\
&= \lfloor \frac{d}{3m}-\frac{1}{3}\rfloor .
\end{align*}
When $d_1< s_L$, \eqref{eqn: matthew} is
\begin{equation}\label{eqn: peter}
\sum_{s=0}^{3m} 2^{\eta(\epsilon)s}+\epsilon \sum_{l=1}^{l_d} \sum_{s= s_{l-1}+1}^{s_l}2^{\eta(\epsilon)s} (4C_2C_3\epsilon^{1/2})^{l-1}+\epsilon \sum_{s=s_{l_d}+1}^{d_1}2^{\eta(\epsilon)s}(4C_2C_3\epsilon^{1/2})^{l_d-1} .
\end{equation}
(If $l_d\le 0$, the second and third terms are zero.) As in the case $s_L\le d_1$, the first summand in \eqref{eqn: peter} is bounded by $\theta(\epsilon)2^{3m}\eta(\epsilon)$.
Multiplying this by $C_2 2^{2d}\pi_3(2^d) 2^{-\eta(\epsilon)d}$ and summing over $d$ from $0$ to $k+1$, we find
%\begin{equation}\label{eqn: sao-paolo}
%2C_2C_32^{2k}\pi_3(2^k) \frac{1}{c_*}\frac{\log \frac{1}{\epsilon}}{\log 2}\epsilon^{-4}\epsilon^{-2c_*M } 2^{-\eta(\epsilon)k}.
%\end{equation}
a bound of
\begin{equation}\label{eqn: corcovado}
C_2C_3 2^{2k}\pi_3(2^k) \epsilon^{16L}, 
\end{equation}
for $k\ge s_{L+1}$.
%, we have
%\[2^{-\eta(\epsilon)k}\le 2^{-c_*M(3L+2)\log \frac{1}{\epsilon}}\le \epsilon^{21 L}\epsilon^{2Mc_*}.\]
%Thus, \eqref{eqn: sao-paolo} is bounded by
%\begin{equation}\label{eqn: corcovado}
%2C_2C_3 2^{2k}\pi_3(2^k) \epsilon^{10L}, 
%\end{equation}
%for $k\ge s_L+2m$.

Using $c_* 3M\log \frac{1}{\epsilon} l_d \le \eta(\epsilon)(d-m)$ and performing a dyadic summation similar to the to one leading to \eqref{eqn: parent}, the second and third terms in \eqref{eqn: peter} are seen to give a contribution bounded by
\begin{equation}
%\begin{split}
%&2\epsilon^2 \sum_{l=1}^{l_d} (4C_2C_3\epsilon^{1/2})^{l-1} \sum_{s= s_{l-1}+1}^{s_l}2^{\eta(\epsilon)s}+\epsilon^2 (4C_2C_3\epsilon^{1/2})^{l_d-1} \sum_{s=s_{l_d}+1}^{d_1}2^{\eta(\epsilon)s}\\
%\le~ & 2\epsilon^2 \theta(\epsilon) \sum_{l=1}^{l_d-1}(4C_2C_3\epsilon^{1/2})^{-1}2^{lc_*3M \log \frac{1}{\epsilon}}2^{l\log 4C_2C_3\epsilon^{1/2}}+\epsilon^2 \theta(\epsilon) 2^{\eta(\epsilon)(d-m)}(4C_2C_3\epsilon^{1/2})^{l_d}\\
%\le~ & 4\epsilon^2 \theta(\epsilon)(4C_2C_3\epsilon^{1/2})^{-1}2^{c_*3Ml_d \log \frac{1}{\epsilon} }2^{l_d \log 4C_2C_3\epsilon^{1/2}}+\epsilon^2 \theta(\epsilon) (4C_2C_3\epsilon^{1/2})^{-1}2^{\eta(\epsilon)(d-m)}2^{l_d \log 4C_2C_3\epsilon^{1/2}}\\
%\le~ & 
3\epsilon \theta(\epsilon) (4C_2C_3\epsilon^{1/2})^{-1}2^{\eta(\epsilon)(d-m)}2^{l_d \log 4C_2C_3\epsilon^{1/2}}. \label{eqn: obediah}
%\end{split}
\end{equation}
%\begin{align*}
%&= 3m\eta(\epsilon) l_d\\
%&= 3m \eta(\epsilon) \lfloor \frac{d}{3m}-\frac{1}{3}\rfloor\\
%&
%\end{align*}

Multiplying \eqref{eqn: obediah} by $C_2 2^{2d}\pi_3(2^d)2^{-\eta(\epsilon) d}$, and adding \eqref{eqn: corcovado}, we find that the contribution to \eqref{eqn: matthew} from $s_L > d_1$ is bounded by
\begin{equation}\label{eqn: severin}
C_2C_3 2^{2k}\pi_3(2^k) \epsilon^{16L} + 3C_2 (4C_2C_3\epsilon^{1/2})^{-1}\theta(\epsilon) \epsilon^{c_*M} \epsilon \sum_{d=m}^{s_L+m}2^{2d}\pi_3(2^d)2^{l_d \log 4C_2C_3\epsilon^{1/2}}.
%=&~2C_2C_3 2^{2k}\pi_3(2^k) \epsilon^{10L}\nonumber \\
%&+ 4\epsilon^3\sum_{d=0}^{s_L+m}2^{2d}\pi_3(2^d)2^{l_d \log 4C_2C_3\epsilon^{1/2}};
 %\label{eqn: severin}
\end{equation}
Note that if
\[ s_{l-1} < d \le s_l,\]
then $l-2\le l_d\le l-1$. The sum in \eqref{eqn: severin} is bounded by
\begin{align*}
&\sum_{d=m}^{s_L+m}2^{2d}\pi_3(2^d)2^{l_d \log 4C_2C_3\epsilon^{1/2}}\\
%=&\sum_{l=1}^L\sum_{s_{l-1}< d\le  s_l}2^{2d}\pi_3(2^d)2^{(l-2) \log 4C_2C_3\epsilon^{1/2}}+\sum_{s_L< d\le s_L+ m}2^{2d}\pi_3(2^d)2^{l_d \log 4C_2C_3\epsilon^{1/2}}\\
%\le& (4C_2C_3\epsilon^{1/2})^{-2}\sum_{l=1}^L (4C_2C_3\epsilon^{1/2})^l C_32^{2\cdot 3m l}\pi_3(2^{3ml})+ C_3 (4C_2C_3\epsilon^{1/2})^{L-1} 2^{2(s_L+m)}\pi_3(2^{s_L+m})\\
\le& 2C_5 C_3 2^{2\cdot 3mL}\pi_3(2^{3mL})(4C_2C_3\epsilon^{1/2})^{L-2}+ C_3 (4C_2C_3\epsilon^{1/2})^{L-1} 2^{2(s_L+m)}\pi_3(2^{s_L+m}).
\end{align*}
Here we have performed a summation as in \eqref{eqn: sum-0}. By \eqref{eqn: chicago}, the pre-factor in front of the sum in \eqref{eqn: severin} is bounded by $4\epsilon^2$ (if $\epsilon$ is small depending on $C_2$), so we find an estimate for the second term of \eqref{eqn: severin} of
\[ 2C_5 2^{2\cdot 3mL}\epsilon \pi_3(2^{3mL})(4C_2C_3\epsilon^{1/2})^L+\epsilon^{3/2} (4C_2C_3\epsilon^{1/2})^L 2^{2(s_L+m)}\pi_3(2^{s_L+m}).\]
Returning to \eqref{eqn: severin}, we find that the contribution to \eqref{eqn: matthew} from $d$ such that $d_1\le s_L$ is bounded by 
\begin{equation}
\label{eqn: vitamin}
C_2C_3 2^{2k}\pi_3(2^k)   \epsilon^{16L}+ 2^{2(3mL+m)}\pi_3(2^{3mL+m})(4C_2C_3\epsilon^{1/2})^L(2C_5^2\epsilon^2+\epsilon^{3/2})).
\end{equation}

Using \eqref{eqn: leslie}, we have:
%\begin{align*}
\[2^{2(3mL+m)}\pi_3(2^{3mL+m}) \le \epsilon^{10}2^{2k}\pi_3(2^k),\]
%&\le C_5 2^{(2-\beta)(3mL+m-k)} 2^{2k}\pi_3(2^k)\\
%&\le C_5 2^{-m} 2^{2k}\pi_3(2^k)\\
%&\le \epsilon^{10}2^{2k}\pi_3(2^k).
when $k\ge s_L+3m$ and $\epsilon$ is small. Putting this into \eqref{eqn: vitamin}, we find a bound of
\begin{equation}\label{eqn: insulin}
\epsilon (4C_2C_3\epsilon^{1/2})^L2^{2k}\pi_3(2^k),
\end{equation}
$k\ge s_L+3m$.

\subsubsection{Reckoning}
Combining \eqref{eqn: jeremiah}, \eqref{eqn: rembrandt} and \eqref{eqn: insulin}, we find for $\epsilon$ small enough,
\begin{equation}\label{eqn: reckoning}
\mathbf{E}[\#\sigma(L) \mid E_k']\le \epsilon (4C_2C_3)^{L+1}\epsilon^{(L-1)/2}2^{2k}\pi_3(2^k),
\end{equation}
for $k\ge s_L+3m$, from which we obtain \eqref{eqn: hans} for $k\ge s_{L+1}$.
\end{proof}

%\begin{figure}
%\centering
%\scalebox{0.6}{\includegraphics[trim={0 0 0 0},clip]{fivearms}}
%\caption{One of the two ``lower boxes'' in the U-shaped region. These have side% length $2^{k+1}/3$.}
%\label{ushaped}
%\end{figure}

\section{Proof of Theorem~\ref{thm: main}}\label{sec: main_proof}
The proof of the main theorem uses a similar but simpler construction to that which appeared in Section~\ref{sec: construction}, and follows that of the main derivation of \cite{DHSchemical1}. For this reason, we omit some details.

Using \eqref{eqn: leslie}, we first choose $\delta>0$ small enough so that for $n$ large, one has 
\begin{equation}\label{eq: delta_choice_final}
n^{1+2\delta} \leq n^2\pi_3(n),
\end{equation} 
and define the truncated box
\[
\hat B(n) = B(n-n^\delta).
\]
This box is chosen so that the total number of edges contained in $B(n) \setminus \hat B(n)$ is at most $Cn^{-\delta} n^2 \pi_3(n)$, and so this region does not significantly contribute to the volume of the lowest crossing $l_n$. Around each $e \in \hat B(n) \cap l_n$, we will search for shortcuts between scales $n^{\delta/8}$ and $n^{\delta/4}$ which give a savings compared to $l_n$ of at least $n^{-c}$ for some $c>0$.

Precisely, from Proposition~\ref{prop: main}, we may choose $a<1$ so that for $\epsilon$ sufficiently small,
\[
\mathbb{E}[\#\mathfrak{s}_j \mid E_j'] \leq a^j 2^{2j}\pi_3(2^j) \quad \text{for all large }n \text{ and } j \in \left( \frac{\delta}{8} \log n, \frac{\delta}{4}\log n\right).
\]
From this we conclude that for $c = \frac{\delta}{8} \log \frac{1}{a}$, one has
\begin{equation}\label{eq: expected_bound_final}
\mathbb{E}[\#\mathfrak{s}_j \mid E_j'] \leq n^{-c} 2^{2j} \pi_3(2^j) \quad \text{for all large }n \text{ and } j \in \left( \frac{\delta}{8} \log n, \frac{\delta}{4} \log n\right)
\end{equation}
%and consequently from \eqref{eqn: triplec-lwr-bd}, for such $n$ and $j$, and any $e$,
%\begin{equation}\label{eq: E_j_bound_final}
%\mathbb{P}(E_j(e,\epsilon,n^{-c})) \geq \frac{c_0c_2}{2} \epsilon^4.
%\end{equation}

We next place $n^{-c}$-shortcuts (as in Definition~\ref{def: shortcuts}) on the lowest crossing in a maximal way, like before. That is, we select a collection of such shortcuts with the property that their detoured paths do not share any vertices, and the total length of their detoured paths is maximal. We then let $\sigma$ be the open path consisting of the union of these shortcuts and the portions of $l_n$ that are left undetoured. Just as in Claim~\ref{claim: comb}, any edge on the lowest crossing that is not covered by such a shortcut in $\sigma$ has no such shortcut around it at all. Because the events $E_k(e) \cap \{e \in l_n\}$ imply existence of shortcuts (Proposition~\ref{prop: implies-S}), one can again place \eqref{eq: expected_bound_final} into Propositions~\ref{prop: concentration} and \ref{prop: changecond} (just as in \eqref{eqn: napoleon}) to find $\eta>0$ such that for all large $n$, and uniformly in $e \subset \hat B(n)$, the probability that an edge $e$ of the lowest crossing is not covered by a detour in $\sigma$ is at most
\begin{align}
\mathbb{P}(\text{there is no }n^{-c} \text{-shortcut around }e \mid e \in l_n) & \leq\mathbb{P}\left( \cap _{j= \lceil \frac{\delta}{8} \log n \rceil}^{\lfloor \frac{\delta}{4} \log n \rfloor} E_j(e,\epsilon,n^{-c})^c \mid e \in l_n\right) \nonumber \\
&\leq 2^{-\frac{\hat c \frac{\delta}{16} \log n}{\log \frac{1}{\epsilon}}} \nonumber \\
&\leq n^{-\eta}. \label{eq: napoleons_last_battle}
\end{align}

Last, we write $(\tau_\ell)$ for the collection of detoured paths in $l_n$ and use \eqref{eq: delta_choice_final} and \eqref{eq: napoleons_last_battle} to estimate the expected length of $\sigma$ for $n$ large as
\begin{align*}
\mathbb{E}[\#\sigma \mid H_n] &\leq Cn^{1+\delta} + n^{-c} \sum_\ell \mathbb{E}[\#\tau_\ell \cap \hat B(n) \mid H_n] \\
&+ \mathbb{E}[ \# \{e \in l_n \cap \hat B(n) : e \text{ has no } n^{-c}\text{-shortcut}\} \mid H_n] \\
&\leq Cn^{1+\delta} + n^{-c} \mathbb{E}[\#l_n \cap \hat B(n) \mid H_n] + n^{-\eta} \mathbb{E}[\#l_n \cap \hat B(n) \mid H_n] \\
&\leq C [n^{-\delta} + n^{-c} + n^{-\eta}] n^2 \pi(n).
\end{align*}
Because $S_n \leq \#\sigma$, this completes the proof of Theorem~\ref{thm: main}.

\section{Large deviation bound conditional on 3 arms}\label{sec: concentration}
%Here we will try to prove the concentration statement. We will consider the three arm event $A_3(2^n)$, that there is an open path from $e_1$ to $\partial B(2^n)$, an open path from $0$ to $\partial B(2^n)$, and a closed dual path from $(1/2)e_1+(1/2)e_2$ to $\partial B(2^n)$, and that all these paths are vertex disjoint. 
Our aim is to give a bound on the conditional probability, given the three-arm event $A_3(2^n)$, that a small number of events $E_k$, which satisfy the probability bound \eqref{eq: conditional_lower_bound}, occur. On $A_3(2^n)$, we will want to have closed dual circuits with defects around the origin to perform decoupling of various events. So let us fix an integer $N \geq 1$ and, given any $k$, let $\mathfrak{C}_k$ be the event that in $A(2^{kN},2^{(k+1)N}) = B(2^{(k+1)N}) \setminus B(2^{kN})$, there is a closed dual circuit with two defects around the origin. Let $\mathfrak{D}_k$ be the event that there is an open circuit with one defect in the same annulus, $A(2^{kN},2^{(k+1)N})$. We will need a large stack of these circuit events to decouple (seven in total), and so we define this compound circuit event for $k \geq 0$ as $\hat{\mathfrak{C}}_k$, the event that the following occur:
\begin{enumerate}
\item for $i=1, 3, 4, 6, 8, 9$, the event $\mathfrak{C}_{10k + i}$ occurs and
\item the event $\mathfrak{D}_{10k}$ occurs.
\end{enumerate}

We will then begin with a sequence of events $(E_k)_{k \geq 0}$ so that 
\begin{enumerate}
\item[A.] $E_k$ depends on the state of edges in $A(2^{kN}, 2^{(k+1)N})$ and
\item[B.] for some constant $C_0>0$, one has for all $n\geq 0$ and integers $k$ with $0 \leq k \leq \frac{n}{10N}-1$,
\begin{equation}\label{eq: conditional_lower_bound}
\mathbf{P}\left(\mathfrak{B}_k ~\bigg|~ A_3(2^n)\right) \geq C_0,
\end{equation}
where
\begin{equation}\label{eq: b_j_def}
\mathfrak{B}_k = \hat{\mathfrak{C}}_k \cap E_{10k+5}.
\end{equation}
\end{enumerate}
In item B, we are requesting that $E_k$ occur, but also that it be surrounded on both sides by the total of seven defected circuits. These circuits will be needed for the ``resetting'' argument.

Define for $N \geq 1$ and $0 \leq n' \leq n$,
\[
I_{n',n} = \left\{ j = \left\lceil \frac{n'}{10N} \right\rceil, \ldots, \left\lfloor \frac{n}{10N} \right\rfloor - 1 : \mathfrak{B}_j \text{ occurs} \right\}.
\]
Note that if $n-n' \geq 40N$, then the range of $j$ specified in $I_{n',n}$ is nonempty.

\begin{thm}\label{thm: concentration}
There exist universal $C_6>0$ and $N_0>0$ such that for any $N \geq N_0$, any $n',n \geq 0$ satisfying $n-n' \geq 40N$, and any events $(E_k)$ satisfying conditions A and B,
\[
\mathbf{P}\left( \#I_{n',n} \leq C_6 C_0 \frac{n-n'}{N} ~\bigg|~ A_3(2^n)\right)\leq \exp\left( - C_6C_0 \frac{n-n'}{N}\right).
\]
\end{thm}

For the proof of Theorem~\ref{thm: concentration}, we first need to verify that conditional on $A_3(2^n)$, many of the events $\hat{\mathfrak{C}}_k$ occur. So for $0 \leq n' \leq n$, we set
\[
J_{n',n} =  \left\{ j= \left\lceil \frac{n'}{10N} \right\rceil, \ldots, \left\lfloor \frac{n}{10N}\right\rfloor -1 :  \hat{\mathfrak{C}}_j \text{ occurs}\right\}.
\]
%Note that the range of $j$ specified in $J_{n',n}$ is nonempty when $n-n' \geq 30N$.
\begin{prop}\label{prop: circuit_jamming}
There exist $C_7>0$ and $N_0 \geq 1$ such that for all $N\geq N_0$ and $n,n'\geq 0$ with $n-n' \geq 40N$,
%for any sequence $(E_k)$ satisfying conditions A and B above,
\[
\mathbf{P}\left(\#J_{n',n} \leq C_7 \frac{n-n'}{N} ~\bigg|~ A_3(2^n)\right) \leq \exp\left( -C_7 (n-n')\right).
\]
\end{prop}
\begin{proof}
For $0 \leq n_1 \leq n_2$, let $A_3(2^{n_1},2^{n_2})$ be the event that there exist three arms from $B(2^{n_1})$ to $\partial B(2^{n_2})$: there are two open paths and one dual closed path, all disjoint, connecting $B(2^{n_1})$ to $\partial B(2^{n_2})$. First note that
\begin{align}
\mathbf{P}\left(\# J_{n',n} \leq C_7 \frac{n-n'}{N}, A_3(2^n)\right) &\leq \mathbf{P}\left( A_3\left(2^{10N\left\lceil \frac{n'}{10N} \right\rceil}\right)\right) \mathbf{P}\left( A_3\left(2^{10N\left\lfloor \frac{n}{10N}\right\rfloor}, 2^n\right) \right) \nonumber \\
&\times \mathbf{P}\left( \bigcap_{m=10\left\lceil \frac{n'}{10N}\right\rceil}^{10\left\lfloor\frac{n}{10N} \right\rfloor - 1}  A_3(2^{mN},2^{(m+1)N}),~\# J_{n',n} \leq C_7\frac{n-n'}{N}\right) \label{eq: first_eq_taco}
\end{align}
By Menger's theorem, for any $m$, the event $A_3(2^{mN},2^{(m+1)N}) \cap \mathfrak{C}_m^c$ implies $A_3(2^{mN},2^{(m+1)N}) \circ A_1(2^{mN},2^{(m+1)N})$, where $\circ$ indicates disjoint occurrence, and $A_1(2^{mN},2^{(m+1)N})$ is the event that there is one open path from $B(2^{mN})$ to $\partial B(2^{(m+1)N})$. By the RSW theorem and Reimer's inequality, there is therefore $\alpha \in (0,1)$ such that
\begin{equation}\label{eq: reimer_app}
\mathbf{P}(A_3(2^{mN},2^{(m+1)N}) \cap \mathfrak{C}_m^c) \leq 2^{-\alpha N} \mathbf{P}(A_3(2^{mN},2^{(m+1)N})).
\end{equation}
Similar reasoning shows that if $A_3(2^{mN}, 2^{(m+1)N}) \cap \mathfrak{D}_m^c$ occurs, then there are three arms as indicated by the $A_3$ event, but one additional closed dual arm crossing this annulus, and we obtain the same bound
\begin{equation}\label{eq: reimer_app_2}
\mathbf{P}(A_3(2^{mN},2^{(m+1)N} \cap \mathfrak{D}_m^c) \leq 2^{-\alpha N} \mathbf{P}(A_3(2^{mN}, 2^{(m+1)N})).
\end{equation}

Using quasimultiplicativity of arm events \cite[Proposition 12]{nolin}, independence, \eqref{eq: reimer_app}, and \eqref{eq: reimer_app_2}, there is a universal $C_8\geq 1$ such that for all $N$ and all $j \geq 0$,
\begin{align}
&\mathbf{P}\left( \cap_{l=0}^9 A_3(2^{(10j+l)N},2^{(10j+l+1)N}) \cap \hat{\mathfrak{C}}_j^c \right) \nonumber \\
\leq~& \sum_{\stackrel{0 \leq r \leq 9}{r \neq 0,2,5,7}} \mathbf{P}\left( \cap_{l=0}^9 A_3(2^{(10j+l)N},2^{(10j+l+1)N}) \cap \mathfrak{C}_{10j+r}^c\right) \nonumber \\
+~&\mathbf{P}\left(\cap_{l=0}^9 A_3(2^{(10j+l)N}, 2^{(10j+l+1)N}) \cap \mathfrak{D}_{10j}^c\right) \nonumber \\
=~& \sum_{\stackrel{0 \leq r \leq 9}{r \neq 0,2,5,7}} \left[ \left( \prod_{\stackrel{0 \leq l \leq 9}{l \neq r}} \mathbf{P}(A_3(2^{(10j+l)N},2^{(10j+l+1)N})) \right) \mathbf{P}\left(A_3(2^{(10j+r)N},2^{(10j+r+1)N}),\mathfrak{C}_{10j+r}^c\right)\right] \nonumber \\
+~& \left( \prod_{1 \leq l \leq 9} \mathbf{P}\left(A_3(2^{(10j+l)N},2^{(10j+l+1)N}\right)\right) \mathbf{P}\left(A_3(2^{10jN},2^{(10j+1)N}), \mathfrak{D}_{10j}^c\right) \nonumber \\
\leq~&7 \cdot 2^{-\alpha N} \prod_{l=0}^9 \mathbf{P}(A_3(2^{(10j+l)N},2^{(10j+l+1)N})) \nonumber \\
\leq~&7 C_8^9 2^{-\alpha N} \mathbf{P}(A_3(2^{10jN},2^{10(j+1)N})). \label{eq: bernoulli_bound}
\end{align}

These observations lead us to realizing the problem as one of concentration using independent variables. For any integer $j$ with $\frac{n'}{10N} \leq j \leq  \frac{n}{10N} -1$, let $X_j$ be the indicator of the event $\cap_{l=0}^9 A_3(2^{(10j+l)N},2^{(10j+l+1)N}) \cap \hat{\mathfrak{C}}_j^c$. Then \eqref{eq: first_eq_taco} implies
\begin{align}
\mathbf{P}\left(\# J_{n',n} \leq C_7\frac{n-n'}{N},A_3(2^n)\right) &\leq \mathbf{P}\left( A_3\left(2^{10N\left\lceil \frac{n'}{10N} \right\rceil}\right)\right) \mathbf{P}\left( A_3\left(2^{10N\left\lfloor \frac{n}{10N}\right\rfloor}, 2^n\right) \right) \nonumber\\
&\times \mathbf{P}\left( \sum_{j=\left\lceil \frac{n'}{10N}\right\rceil}^{\left\lfloor \frac{n}{10N} \right\rfloor -1} X_j \geq \left\lfloor \frac{n}{10N} \right\rfloor - \left\lceil \frac{n'}{10N}\right\rceil - C_7 \frac{n-n'}{N} \right). \label{eq: master_1}
\end{align}
Using \eqref{eq: bernoulli_bound} and the RSW theorem, the $X_j$'s are independent Bernoulli random variables with parameters  $p_j$ that satisfy for some $\beta\geq 1$
\begin{equation}\label{eq: p_j_bound}
2^{-\beta N} \leq p_j \leq 7C_8^9 2^{-\alpha N}\mathbf{P}(A_3(2^{10jN},2^{10(j+1)N})).
\end{equation}
So we need an elementary lemma about concentration of independent Bernoulli random variables with suitable parameters. 
%{\color{red} next lemma can be made better but appears to be unnecessary}
\begin{lma}\label{lem: stirling}
Given $\epsilon_1 \in (0,1)$ and $M\ge 1$, if $Y_1, \ldots, Y_M$ are any independent Bernoulli random variables with parameters $p_1, \ldots, p_M$ respectively satisfying $p_i \in [\epsilon_1,1]$ for all $i$, then for all $r \in (0,1)$,
\[
\mathbf{P}\left( \sum_{i=1}^M Y_i \geq rM \right) \leq (1/\epsilon_1)^{M(1-r)} 2^M \prod_{i=1}^M p_i.
\]
\end{lma}
\begin{proof}
One has
\[
\mathbf{P}(Y_1 + \cdots + Y_M \geq rM) = \sum_{\ell = \lceil rM \rceil}^M \mathbf{P}(Y_1 + \cdots + Y_M = \ell ).
\]
Also for $\ell$ with $\lceil rM \rceil \leq \ell \leq M$,
\begin{align*}
\mathbf{P}(Y_1 + \cdots + Y_M = \ell ) &= \sum_{\stackrel{y_1, \ldots, y_M \in \{0,1\}}{y_1 + \cdots + y_M = \ell }} p_1^{y_1} \cdots p_M^{y_M} (1-p_1)^{1-y_1} \cdots (1-p_M)^{1-y_M} \\
&= \prod_{i=1}^M p_i \sum_{\stackrel{y_1, \ldots, y_M \in \{0,1\}}{y_1 + \cdots + y_M = \ell }} \left( \frac{1-p_i}{p_i}\right)^{1-y_i} \\
&\leq \binom{M}{\ell}  \left( \frac{1-\epsilon_1}{\epsilon_1} \right)^{M-\ell} \prod_{i=1}^M p_i \\
&\leq \binom{M}{\ell} (1/\epsilon_1)^{M(1-r)} \prod_{i=1}^M p_i.
\end{align*}
We sum over $\ell$ to obtain
\[
\mathbf{P}(Y_1 + \cdots + Y_M \geq rM) \leq ( 1/\epsilon_1 )^{M(1-r)}  \left(\sum_{\ell = \lceil r M \rceil}^M \binom{M}{\ell} \right)\prod_{i=1}^M p_i,
\]
from which the lemma follows.
%To bound the sum of binomial coefficients, we write it as $2^M \sum_{\ell \geq rM}^M \binom{M}{\ell} 2^{-M} = 2^M \mathbf{P}(Z_1 + \cdots + Z_M \geq rM)$, where $(Z_i)$ are i.i.d. Bernoulli$(1/2)$. Because the sequence $(a_M)$ given by $a_M = \mathbf{P}(Z_1 + \cdots + Z_M \geq rM)$ is supermultiplicative, one has $a_M \leq e^{-\lambda M}$ for all $M \geq 1$, where $\lambda = \lim_M - \frac{1}{M} \log a_M$. It is standard to compute $\lambda = r \log r + (1-r)\log (1-r) + \log 2$, and so
%\[
%\sum_{\ell = \lceil rM \rceil}^M \binom{M}{\ell} \leq \exp\left( - M (r \log r + (1-r)\log(1-r))\right) \text{ for all } M \geq 1.
%\]
%Placing this in \eqref{eq: pre_eq} completes the proof of the lemma.
%{\color{red} start here. maybe easier to relate to fair coin flips} For the binomial coefficient, we can use Stirling's formula to show that one has
%\[
%\binom{M}{\lceil r M \rceil} \sim \frac{1}{\sqrt{2\pi r(1-r)}} \exp\left( - M (r \log r + (1-r)\log (1-r))\right) \text{ as }M \to \infty.
%\]
%As $r \uparrow 1$, the exponent $-(r \log r + (1-r)\log(1-r))$ becomes arbitrarily close to 0, so for a given $\delta$ and $\epsilon_1$, we can choose $r$ close enough to 1 and a constant $C_3>0$ such that
%\[
%\binom{M}{\lceil rM \rceil} \leq C_3 (1/\epsilon_1)^{\delta M} \text{ for all } M.
%\]
%As for the next term of \eqref{eq: pre_eq}, for $C_3$ chosen appropriately, one has $2M \leq C_3(1/\epsilon_1)^{\delta M}$. Furthermore, if $r$ is close enough to $1$,
%\[
%(1/\epsilon_1)^{M-\lceil rM \rceil} \leq (1/\epsilon_1)^{\delta M}.
%\]
\end{proof}

We now apply Lemma~\ref{lem: stirling} to \eqref{eq: master_1}, using the bounds from \eqref{eq: p_j_bound}, with $\epsilon_1 = 2^{-\beta N}$. Note that if $C_7 < 1/20$ and $n-n' \geq 40N$, one has
\[
\left\lfloor \frac{n}{10N} \right\rfloor - \left\lceil \frac{n'}{10N} \right\rceil - C_7 \frac{n-n'}{N} \geq \left( \left\lfloor \frac{n}{10N} \right\rfloor - \left\lceil \frac{n'}{10N} \right\rceil \right)(1-20C_7).
\]
So if we put $r=1-20C_7$ (noting that $r \in (0,1)$) and use Lemma~\ref{lem: stirling}, we continue from \eqref{eq: master_1} to obtain
\begin{align}
&\mathbf{P}\left(\#J_{n',n} \leq C_7 \frac{n-n'}{N}, A_3(2^n)\right) \nonumber \\
\leq~& \mathbf{P}\left( A_3\left(2^{10N\left\lceil \frac{n'}{10N} \right\rceil}\right)\right) \mathbf{P}\left( A_3\left(2^{10N\left\lfloor \frac{n}{10N}\right\rfloor}, 2^n\right) \right) \nonumber \\
\times~&\mathbf{P}\left( \sum_{j=\left\lceil \frac{n'}{10N}\right\rceil}^{\left\lfloor \frac{n}{10N} \right\rfloor -1} X_j \geq \left(\left\lfloor \frac{n}{10N} \right\rfloor - \left\lceil \frac{n'}{10N}\right\rceil\right) (1-20C_7) \right) \nonumber\\
\leq~& (2^{\beta N})^{\left(\left\lfloor \frac{n}{10N} \right\rfloor - \left\lceil \frac{n'}{10N} \right\rceil\right) \cdot 20C_7} 2^{\left\lfloor \frac{n}{10N} \right\rfloor - \left\lceil \frac{n'}{10N} \right\rceil}  \nonumber\\
\times~&\mathbf{P}\left( A_3\left(2^{10N\left\lceil \frac{n'}{10N} \right\rceil}\right)\right) \mathbf{P}\left( A_3\left(2^{10N\left\lfloor \frac{n}{10N}\right\rfloor}, 2^n\right) \right)\prod_{j=\left\lceil \frac{n'}{10N} \right\rceil}^{\left\lfloor \frac{n}{10N}\right\rfloor-1} 7C_8^9 2^{-\alpha N} \mathbf{P}(A_3(2^{10jN},2^{10(j+1)N})) \nonumber\\
=~& \left( 14C_8^9 2^{(20\beta C_7 - \alpha) N} \right)^{\left\lfloor \frac{n}{10N} \right\rfloor - \left\lceil \frac{n'}{10N} \right\rceil}   \nonumber\\
\times~&\mathbf{P}\left( A_3\left(2^{10N\left\lceil \frac{n'}{10N} \right\rceil}\right)\right) \mathbf{P}\left( A_3\left(2^{10N\left\lfloor \frac{n}{10N}\right\rfloor}, 2^n\right) \right)\prod_{j=\left\lceil \frac{n'}{10N} \right\rceil}^{\left\lfloor \frac{n}{10N}\right\rfloor-1} \mathbf{P}(A_3(2^{10jN},2^{10(j+1)N})) \label{eq: taco_head_supreme}
\end{align}

%C_3 2^{\beta N \delta \left\lfloor \frac{n}{5N} \right\rfloor} \prod_{j=0}^{\left\lfloor \frac{n}{5N} \right\rfloor-1} \left(5C_2^4 2^{-\alpha N}\mathbf{P}(A_3(2^{5jN},2^{5(j+1)N}))\right) \\
%&\times \mathbf{P}\left( A_3(2^{5N\left\lfloor \frac{n}{5N}\right\rfloor}, 2^n)\right) \\
%&= C_3 \left( 5 C_2^4 2^{(\beta \delta - \alpha) N}\right)^{\left\lfloor \frac{n}{5N} \right\rfloor} \prod_{j=0}^{\left\lfloor \frac{n}{5N} \right\rfloor-1} \mathbf{P}(A_3(2^{5jN},2^{5(j+1)N})) \\
%&\times \mathbf{P}\left( A_3(2^{5N\left\lfloor \frac{n}{5N}\right\rfloor}, 2^n)\right).
%\end{align*}
Again by quasimultiplicativity of arm events,
\begin{align*}
\mathbf{P}\left( A_3\left(2^{10N\left\lceil \frac{n'}{10N} \right\rceil}\right)\right) \mathbf{P}\left( A_3\left(2^{10N\left\lfloor \frac{n}{10N}\right\rfloor}, 2^n\right) \right)\prod_{j=\left\lceil \frac{n'}{10N} \right\rceil}^{\left\lfloor \frac{n}{10N}\right\rfloor-1} &\mathbf{P}(A_3(2^{10jN},2^{10(j+1)N})) \\
&\leq C_8^{\left\lfloor \frac{n}{10N} \right\rfloor - \left\lceil \frac{n'}{10N} \right\rceil + 1} \mathbf{P}(A_3(2^n)).
\end{align*}
Use this estimate in \eqref{eq: taco_head_supreme} to find for $C_7 < 1/20$, $n-n' \geq 40N$, and all $N \geq 1$,
\begin{align*}
\mathbf{P}\left(\#J_{n',n} \leq C_7 \frac{n-n'}{N} ~\bigg|~ A_3(2^n)\right) &\leq C_8 \left( 14C_8^{10} 2^{(20\beta C_7 - \alpha) N} \right)^{\left\lfloor \frac{n}{10N} \right\rfloor - \left\lceil \frac{n'}{10N} \right\rceil} \\
&\leq \left( 14C_8^{11} 2^{(20\beta C_7 - \alpha) N} \right)^{\left\lfloor \frac{n}{10N} \right\rfloor - \left\lceil \frac{n'}{10N} \right\rceil}
\end{align*}
Lowering $C_7$ so that $C_7 < \alpha/(40\beta)$, one has $20\beta C_7 - \alpha \leq -\alpha/2$. We also pick $N_0$ so large that for $N \geq N_0$, one has $14C_8^{11} \leq 2^{\alpha N/4}$ and obtain
\[
\mathbf{P}\left(\#J_{n',n} \leq C_7 \frac{n-n'}{N} ~\bigg|~ A_3(2^n)\right) \leq 2^{-\alpha \frac{N}{4} \left( \left\lfloor \frac{n}{10N} \right\rfloor - \left\lceil \frac{n'}{10N} \right\rceil \right)}
\]
If $n-n' \geq 40N$, then we obtain the upper bound $2^{-\alpha (n-n')/80}$, which completes the proof of Proposition~\ref{prop: circuit_jamming}.
\end{proof}

Given the bound on the probability of existence of many decoupling circuits from Proposition~\ref{prop: circuit_jamming}, we move to the proof of Theorem~\ref{thm: concentration}.
\begin{proof}[Proof of Theorem~\ref{thm: concentration}]
For $N \geq N_0$ and $n,n' \geq 0$ such that $n-n' \geq 40N$, we will estimate $\# I_{n',n}$ using the standard Chernoff bound along with a decoupling argument. So estimate using Proposition~\ref{prop: circuit_jamming}, for $C_9>0$ to be determined at the end of the proof,
\begin{align}
&\mathbf{P}\left( \# I_{n',n} \leq C_9 \frac{n-n'}{N}~\bigg|~ A_3(2^n) \right) \nonumber \\
\leq~& \exp\left( - C_7(n-n')\right) + \mathbf{P}\left( \#I_{n',n} \leq C_9\frac{n-n'}{N}, \#J_{n',n} \geq C_7 \frac{n-n'}{N} ~\bigg|~ A_3(2^n)\right) \nonumber \\
\leq~& \exp\left(- C_7(n-n')\right) + \exp\left( C_9 \frac{n-n'}{N}\right) \EE\left[e^{-\#I_{n',n}} \mathbf{1}_{\{\#J_{n',n} \geq C_7 \frac{n-n'}{N}\}} ~\bigg|~ A_3(2^n)\right]. \label{eq: come_here!}
\end{align}
The expectation we decompose over all possible sets $J_{n',n}$ as
\begin{equation}\label{eq: decompose_over_j}
\sum_{\#\mathcal{J} \geq C_7 \frac{n-n'}{N}} \EE\left[ e^{- \#I_{n',n}} ~\bigg|~J_{n',n} = \mathcal{J},~A_3(2^n)\right] \mathbf{P}(J_{n',n}= \mathcal{J} \mid A_3(2^n)).
\end{equation}
Last, we expand the expectation over a filtration. Enumerate the set $\mathcal{J} = \{j_1, \ldots, j_{r_0}\}$, where ${r_0}\geq C_7 \frac{n-n'}{N}$. Then a.s. relative to the measure 
\[
\hat{\mathbf{P}} := \mathbf{P}\left( \cdot \mid J_{n',n} = \mathcal{J}, A_3(2^n)\right),
\] 
one has $\#I_{n',n} = \sum_{s=1}^{r_0} \mathbf{1}_{\{E_{10j_s+5}\}}$. For fixed $\mathcal{J}$, define the filtration $(\mathcal{F}_s)$ by
\[
\mathcal{F}_s = \sigma\left\{ E_{10j_1+5}, \ldots, E_{10j_{s-1}+5}\right\} \text{ for } s = 1, \ldots, {r_0}.
\]
(Here, $\mathcal{F}_1$ is trivial.) Now the expectation in \eqref{eq: decompose_over_j} can be written using the expectation $\hat{\EE}$ relative to $\hat{\mathbf{P}}$ as
\begin{equation}\label{eq: cond_exp_expansion}
\hat{\EE}\left[ e^{- \mathbf{1}_{E_{10j_1+5}}} \cdots \hat{\EE}\left[ e^{- \mathbf{1}_{E_{10j_{s-1}+5}}} \hat{\EE}\left[ e^{- \mathbf{1}_{E_{10j_{r_0}+5}}} ~\bigg|~ \mathcal{F}_{r_0}\right] ~\bigg|~ \mathcal{F}_{s-1} \right] \cdots ~\bigg|~ \mathcal{F}_1 \right].
\end{equation}
For any $s=1, \ldots, {r_0}$, one has $\hat{\mathbf{P}}$-a.s., 
\begin{equation}\label{eq: before_decoupling}
\hat{\EE}\left[ e^{- \mathbf{1}_{E_{10j_s+5}}} ~\bigg|~\mathcal{F}_s\right] = 1- \hat{\mathbf{P}}(E_{10j_s+5} \mid \mathcal{F}_s)(1-e^{-1}).
\end{equation}

We bound this conditional expectation uniformly over $s$ and $\omega$ using the following decoupling estimate.

\begin{lma}\label{lem: FGdecouple}
There exists a universal constant $c_1 > 0$ such that the following holds. For any $k,n \geq 0$ and $N \geq 1$ satisfying
\[ 
k  \leq \left\lfloor \frac{n}{10N} \right\rfloor - 1,
\]  
and any events $F$ and $G$ depending on the status of edges in $B(2^{10kN})$ and $B(2^{10(k+1)N})^c$ respectively, one has
\begin{equation}
\label{eq:decmain2}
\mathbf{P}\left(E_{10k+5} \mid \hat{\mathfrak{C}}_k ,\, A_3(2^n),\, F, G \right) \geq c_1 \mathbf{P}\left(E_{10k+5} \mid \hat{\mathfrak{C}}_k,\, A_3(2^n) \right). 
\end{equation}
\end{lma}
\begin{proof}

We first prove a partial version of Lemma~\ref{lem: FGdecouple}, where we remove the conditioning on $F$ but not $G$: under the assumptions of Lemma~\ref{lem: FGdecouple}, one has
  \begin{equation}
    \label{eq:decmain}
    \mathbf{P}\left(E_{10k+5} \mid \hat{\mathfrak{C}}_k ,\, A_3(2^n),\, F, G \right) \geq c_2 \mathbf{P}\left(E_{10k+5} \mid \hat{\mathfrak{C}}_k,\, A_3(2^n),\, G \right).
  \end{equation}

  The proof of \eqref{eq:decmain} proceeds via decoupling using the block of circuits whose existence is guaranteed by $\hat{\mathfrak{C}}_k$. For $\ell = 1,\, 4,\, 6,\,9$ and an outcome in $A_3(2^n)$, let $Circ_\ell(\Cc)$ be the event that $\Cc$ is the innermost (vertex self-avoiding) closed dual circuit with exactly two defects in $A(2^{(10k+\ell)N},\, 2^{(10k+\ell+1)N})$. If $A_3(2^n)$ does not occur, $Circ_\ell(\Cc)$ is the event that $\Cc$ is a closed dual circuit with exactly two open defects in $A(2^{(10k + \ell)N},\,2^{(10k + \ell + 1)N})$, such that no other such circuit in this annulus is contained in the union of $\Cc$ and its interior.

Conditioning on $F$ can change the probabilities of the various $Circ_{1}(\Cc)$ events. The role of the outer defected dual circuit (from $\mathfrak{C}_{10k+4}$) appearing before $E_{10k+5}$ is to approximately remove this bias introduced by $F$. We make this decoupling explicit by breaking the intersection on the left-hand side of \eqref{eq:decmain} into several pieces. 
%First, for any $n' \leq m_1 < m_2 \leq n$, let $\Jc_{m_1, m_2}$ denote the subset $\Jc \cap [m_1, m_2]$ so that $\Jc_{m_1, m_2}$ is a possible realization of $J_{m_1, m_2}$.

Any closed dual circuit $\Cc$ with exactly two defects has two disjoint closed arcs between these defects; order all defects and arcs arbitrarily and number the defects (resp. arcs) of $\Cc$ according to this ordering as $e_i(\Cc)$ (resp. $\Ac_i(\Cc)$) for $i = 1,\, 2$. For $\Cc$ a closed dual circuit with two defects in $A(2^{(10k+1)N}, 2^{(10k+2)N})$, let $X_{-}(\Cc,\, i)$ denote the event that
\begin{enumerate}
\item $\mathfrak{D}_k \cap Circ_{1}(\Cc)$ occurs;
\item the edge $\{0,\mathbf{e}_1\}$ is connected to $e_1(\Cc)$ and $e_2(\Cc)$ in the interior of $\Cc$ via vertex-disjoint open paths;
\item $\frac{1}{2}(\mathbf{e}_1 + \mathbf{e}_2)$ is connected to $\Ac_i(\Cc)$ via a closed dual path.
%\item $J_{n', k-1} = \Jc_{n', k-1}$.
\end{enumerate}

We first make the following claim, which will be useful in decomposing the events appearing in \eqref{eq:decmain}:
\begin{equation}\label{eq: uniquearc}
  \text{On }A_3(2^n) \cap \hat{\mathfrak{C}}_k, \text{ the event }X_{-}(\Cc,\, i) \text{ occurs for exactly one choice of }\Cc \text{ and }i. 
\end{equation}
We omit the proof of \eqref{eq: uniquearc}; the essential point is the presence of the  open defected circuit in $A(2^{10kN},\, 2^{(10k+1)N})$ having exactly one closed defect. This guarantees that exactly one $\Ac_i(\Cc)$ can connect to $\frac{1}{2}(\mathbf{e}_1+\mathbf{e}_2)$, since any closed path from the aforementioned defect will be confined by a pair of disjoint open paths leading to $e_1(\Cc)$ and $e_2(\Cc)$.

We will decompose $\hat{\mathfrak{C}}_k$ into inner, outer, and middle pieces; the above gives the ``inner'' piece. To build the outer piece, let $\hat{\mathfrak{C}}_k^+$ be the event that $\mathfrak{C}_{10k+\ell}$ occurs for $\ell = 6,\, 8,\, 9$.
Similarly, to the above, let $\Dc$ be a dual circuit in $A(2^{(k+4)N},2^{(k+5)N})$ (it will eventually be taken closed with two defects) with two distinguished primal edges $\{e_i(\Dc)\}_{i=1,2}$ crossing it and corresponding arcs $\Ac_i(\Dc)$ between them. We define the event $X_+(\Dc,\, j)$ by the following conditions:
\begin{enumerate}
%\item $Circ_{k-1}(\Dc)$ occurs;
\item $e_1(\Dc)$ and $e_2(\Dc)$ are connected to $\partial B(2^n)$ in the exterior of $\Dc$ via disjoint open paths;
\item $\Ac_j(\Dc)$ is connected in the exterior of $\Dc$ to $\partial B(2^n)$ via a closed dual path;
\item $\hat{\mathfrak{C}}_k^+$ occurs.
%\item $J_{k+1, n} = \Jc_{k+1,n}$.
\end{enumerate}

We also need the probability of ``transitions'' between $\Cc$ and $\Dc$, and it is with these that we implement the decoupling from $F$. For $\Cc$ and $\Dc$ marked dual circuits in annuli as above, let $P(\Cc, \ \Dc,\,i,\,j)$ be the probability, conditional on the event that each $e_i(\Cc)$ is open and all other edges of $\Cc$ are closed, that
\begin{enumerate}
\item $Circ_{4}(\Dc)$ occurs;
\item There is a pair of disjoint open paths in the region between $\Cc$ and $\Dc$ connecting $e_1(\Cc)$ to one of the marked edges $\{e_1(\Dc),\,e_2(\Dc)\}$ and $e_2(\Cc)$ to the other marked edge of $\Dc$;
\item There is a closed dual path in the region between $\Cc$ and $\Dc$ connecting $\Ac_i(\Cc)$ to $\Ac_{j}(\Dc)$;
\item $\mathfrak{C}_{10k+3}$ occurs.
\end{enumerate}

Note that, conditioning on $X_{-}(\Cc,\,i)$ (and further conditioning on events depending on the status of edges in the interior of $\Cc$), the process outside $\Cc$ remains a free percolation. Conditioning also on $X_+(\Dc,\,j)$  and on any other events in the exterior of $\Dc$ leaves free percolation between $\Cc$ and $\Dc$. We last note that if 
$X_{-}(\Cc,\,i)\cap X_+(\Dc,\,j)$ occurs and if the defects of $\Cc$ and $\Dc$ are connected as in item 2 in the definition of $P(\cdot,\cdot,\cdot,\cdot)$, then $\Ac_i(\Cc)$ is connected in the region between $\Cc$ and $\Dc$ to at most one of $\{\Ac_1(\Dc),\,\Ac_2(\Dc)\}$. This follows by another trapping argument involving the open paths.
%\[X_{-}(\Cc,\,i,\,\Jc_{n',k-1})\cap X_+(\Dc,\,j,\,\Jc_{k+1,n}) \cap A_3(2^n) \cap Circ_{3}(\Dc)\]
%occurs, there must be a closed dual connection from $\Ac_i(\Cc)$ to $\Ac_j(\Dc)$ lying entirely between $\Cc$ and $\Dc$ (otherwise)

Using the observations of the above paragraph and \eqref{eq: uniquearc}, we see that for events $E_{10k+5},\,F,\,G$ as in the statement of the proposition:
\begin{align}
  \mathbf{P}&\left(E_{10k+5},\, \hat{\mathfrak{C}}_k,\, A_3(2^n),\,F,\,G \right)\nonumber\\
& = \sum_{\Cc,\,\Dc,\,i,\,j}\mathbf{P}\left(F,\,X_{-}(\Cc,\,i) \right)P(\Cc,\,\Dc,\,i,\,j)\mathbf{P}\left(E_{10k+5},\,X_+(\Dc,\,j),\,G\right).\label{eq:withf}
\end{align}
Similarly, we can decompose 
\begin{align}
  \mathbf{P}&\left(\hat{\mathfrak{C}}_k,\, A_3(2^n),\,G \right)\nonumber\\
& = \sum_{\Cc',\,\Dc',\,i',\,j'}\mathbf{P}\left(\,X_{-}(\Cc',\,i') \right)P(\Cc',\,\Dc',\,i',\,j')\mathbf{P}\left(X_+(\Dc',\,j'),\, G\right),\label{eq:withoutf}
\end{align}
and analogous decompositions hold for other quantities similar to $\mathbf{P}(\hat{\mathfrak{C}}_k,\, A_3(2^n),\,G)$.

To accomplish the decoupling, we use the following inequality which is adapted from, and whose proof is essentially the same as, \cite[Lemma 6.1]{DS} (see also \cite[Lemma 23]{DHSchemical1}). It gives a form of comparability for the various circuit transition factors.
%\begin{lma}\label{lma:circdecoup}
There exists a uniform constant $C_{10} < \infty$ such that the following holds uniformly in $k,N$, as well as in choices of circuits $\Cc,\,\Cc',\,\Dc,\,\Dc'$ and arc indices $i,\,j,\,i',\,j'$:
\begin{equation}\label{eq: circdecoup}
\frac{P(\Cc,\,\Dc,\,i,\,j) P(\Cc',\,\Dc',\,i',\,j')}{P(\Cc,\,\Dc',\,i,\,j') P(\Cc',\,\Dc,\,i',\,j)} < C_{10}. 
\end{equation}
%\end{lma}

To apply \eqref{eq: circdecoup}, multiply \eqref{eq:withf} and \eqref{eq:withoutf}:
\begin{align}
  \mathbf{P}&\left(E_{10k+5},\, \hat{\mathfrak{C}}_k,\, A_3(2^n),\, F,\,G \right) \mathbf{P}\left( \hat{\mathfrak{C}}_k,\,\, A_3(2^n),\, G \right)\nonumber\\
  &= \sum_{\substack{\Cc,\,\Dc,\,i,\,j\\ \Cc',\,\Dc',\,i',\,j'}}\Big[\mathbf{P}\left(\,X_{-}(\Cc',\,i') \right)P(\Cc',\,\Dc',\,i',\,j')\mathbf{P}\left(X_+(\Dc',\,j'),\, G\right) \nonumber\\
  &\qquad \times\mathbf{P}\left(F,\,X_{-}(\Cc,\,i) \right)P(\Cc,\,\Dc,\,i,\,j)\mathbf{P}\left(E_{10k+5},\,X_+(\Dc,\,j),\,G\right) \Big]\label{eq:expand1}\\
&\geq C_{10}^{-1} \sum_{\substack{\Cc,\,\Dc,\,i,\,j\\ \Cc',\,\Dc',\,i',\,j'}}\Big[\mathbf{P}\left(\,X_{-}(\Cc',\,i') \right)P(\Cc',\,\Dc,\,i',\,j) \mathbf{P}\left(E_{10k+5},\,X_+(\Dc,\,j),\, G\right)\nonumber\\
  &\qquad \times\mathbf{P}\left(F,\,X_{-}(\Cc,\,i) \right)P(\Cc,\,\Dc',\,i,\, j') \mathbf{P}\left(X_+(\Dc',\,j'),\,G\right) \Big]\nonumber\\
&= C_{10}^{-1}\mathbf{P}\left(E_{10k+5},\, A_3(2^n),\, \hat{\mathfrak{C}}_k,\,G \right) \mathbf{P}\left(A_3(2^n),\, \hat{\mathfrak{C}}_k,\,F,\,G\right).\nonumber
\end{align}
Dividing both sides of the above by $\mathbf{P}(\hat{\mathfrak{C}}_k, A_3(2^n),\,G)$ and $\mathbf{P}\left(A_3(2^n),\, \hat{\mathfrak{C}}_k,\,F,\,G\right)$ gives
\begin{align*}
  \mathbf{P}\left(E_{10k+5} \mid A_3(2^n),\,\hat{\mathfrak{C}}_k ,\,F,\,G\right) \geq C_{10}^{-1} \mathbf{P}\left(E_{10k+5} \mid A_3(2^n),\,\hat{\mathfrak{C}}_k,\,G\right).
\end{align*}
This is the claim of \eqref{eq:decmain} with $c_2 = C_{10}^{-1}$.
%\end{proof}

%\begin{proof}[Sketch of proof of Proposition \ref{prop:FGdecouple}]
Equation~\eqref{eq:decmain} allows us to first remove the conditioning on $F$, and using it, we see that to prove Lemma~\ref{lem: FGdecouple}, it suffices to show the existence of a uniform $c_3 > 0$ such that
  \begin{equation}
    \label{eq:onegone}
    \mathbf{P}\left(E_{10k+5} \mid \hat{\mathfrak{C}}_k ,\, A_3(2^n),\, G \right) \geq c_3 \mathbf{P}\left(E_{10k+5} \mid \hat{\mathfrak{C}}_k,\, A_3(2^n) \right).
  \end{equation}
To show \eqref{eq:onegone}, we argue nearly identically to the proof of \eqref{eq:decmain}.  The main difference is just the placement of the circuits and connections in the decoupling. We now have to condition on the values of innermost defected circuits in $A(2^{(k+6)N},2^{(k+7)N})$ and $A(2^{(k+9)N},2^{(k+10)N})$.

Just as before, the effect of conditioning on $G$ is just to bias the distribution of circuits in  $A(2^{(k+9)N},2^{(k+10)N})$, and \eqref{eq: circdecoup} shows that the inner circuit approximately removes this bias. Expanding the product 
\[\mathbf{P}\left(E_{10k+5},\, \hat{\mathfrak{C}}_k,\, A_3(2^n),\,G \right) \mathbf{P}\left( \hat{\mathfrak{C}}_k,\,\, A_3(2^n) \right)  \]
similarly to \eqref{eq:expand1} and regrouping terms after applying \eqref{eq: circdecoup}, Lemma~\ref{lem: FGdecouple} follows.
\end{proof}

Returning to the proof of Theorem~\ref{thm: concentration}, we apply Lemma~\ref{lem: FGdecouple} to prove the following statement.
%\begin{lma}\label{lem: decoupling_lemma}
There exists a universal $C_{11}>0$ such that for any $N \geq 1$, any $n',n \geq 0$ satisfying $n-n' \geq 40N$, any $j = \left\lceil \frac{n'}{10N}\right\rceil, \ldots, \left\lfloor \frac{n}{10N} \right\rfloor - 1$, any $F$ depending on the state of edges in $B(2^{10jN})$, and any $\mathcal{J}$ containing $j$,
\begin{equation}\label{eq: decoupling_equation}
\mathbf{P}(\mathfrak{B}_j \mid F, J_{n',n} = \mathcal{J},A_3(2^n)) \geq C_{11}\mathbf{P}(\mathfrak{B}_j \mid A_3(2^n)).
\end{equation}
To show \eqref{eq: decoupling_equation}, write $\{J_{n',n} = \mathcal{J}\}$ as an intersection $\hat F \cap \hat{\mathfrak{C}}_j \cap G$, where $\hat F$ depends on the state of edges in $B(2^{10jN})$ and $G$ depends on the state of edges in $B(2^{10(j+1)N})^c$. Applying Lemma~\ref{lem: FGdecouple} using $F \cap \hat F$ in place of $F$, we obtain
\begin{align*}
\mathbf{P}(\mathfrak{B}_j \mid F, J_{n',n} = \mathcal{J}, A_3(2^n)) &= \mathbf{P}(E_{10j+5} \mid F, \hat F, \hat{\mathfrak{C}}_j, G, A_3(2^n)) \\
&\geq c_1 \mathbf{P}(E_{10j+5} \mid \hat{\mathfrak{C}}_j, A_3(2^n)) \\
&\geq c_1 \mathbf{P}(\mathfrak{B}_j \mid A_3(2^n)),
\end{align*}
which is \eqref{eq: decoupling_equation} with $C_{11} = c_1$.
%\end{lma}
%\begin{proof}
%{\color{red} INSERT JACK'S PROOF}
%\end{proof}

We now apply \eqref{eq: decoupling_equation} to the probability in \eqref{eq: before_decoupling}. For a fixed $\mathcal{J} = \{j_1, \ldots, j_{r_0}\}$ with ${r_0} \geq C_1 \frac{n-n'}{N}$ and $s=1,\ldots, {r_0}$, let $x_1, \ldots, x_{s-1} \in \{0,1\}$ and put
\[
F = \{\mathbf{1}_{E_{10j_1+5}}=x_1, \ldots, \mathbf{1}_{E_{10j_{s-1}+5}}=x_{s-1}\}.
\]
Then for $\omega \in F \cap \{J_{n',n} = \mathcal{J}\} \cap A_3(2^n)$, the event $\hat{\mathfrak{C}}_{j_s}$ occurs, and so
\begin{align*}
\hat{\mathbf{P}}(E_{10j_s+5} \mid \mathcal{F}_s)(\omega) &= \mathbf{P}(E_{10j_s+5} \mid F,J_{n',n}=\mathcal{J},A_3(2^n)) \\
&= \mathbf{P}(\mathfrak{B}_{j_s} \mid F, J_{n',n} = \mathcal{J}, A_3(2^n)) \\
&\geq C_{11} \mathbf{P}(\mathfrak{B}_{j_s} \mid A_3(2^n)).
\end{align*}
Using assumption \eqref{eq: conditional_lower_bound}, we obtain $\hat{\mathbf{P}}$-a.s. for $\omega \in F \cap \{J_{n',n} = \mathcal{J}\} \cap A_3(2^n)$
\[
\hat{\mathbf{P}}(E_{10j_s+5} \mid \mathcal{F}_s) \geq C_{11}C_0.
\]
Because such events generate the sigma-algebra $\mathcal{F}_s$, the same inequality is valid $\hat{\mathbf{P}}$-a.s., and so replacing this in \eqref{eq: before_decoupling}, we have
\[
\hat{\EE} \left[ e^{-\mathbf{1}_{E_{10j_s+5}}} ~\bigg|~\mathcal{F}_s \right] \leq 1 - C_{11}C_0 (1-e^{-1}).
\]

Starting with this bound for $s={r_0}$, we place it in \eqref{eq: cond_exp_expansion}, and then repeat for $s={r_0}-1$, and so on, until $s=1$ to obtain the overall bound for ${r_0}=\# \mathcal{J} \geq C_7 \frac{n-n'}{N}$
\begin{align*}
\EE\left[ e^{-\# I_{n',n}} ~\bigg|~ J_{n',n}= \mathcal{J},A_3(2^n) \right] &\leq \left( 1-C_{11}C_0 (1-e^{-1})\right)^{r_0} \\
&\leq \left( 1-C_{11}C_0 (1-e^{-1})\right)^{C_7 \frac{n-n'}{N}}.
\end{align*}
We sum this in \eqref{eq: decompose_over_j} for
\begin{align*}
&\EE\left[ e^{-\#I_{n',n}} \mathbf{1}_{\left\{ \# J_{n',n} \geq C_7 \frac{n-n'}{N}\right\}} ~\bigg|~A_3(2^n) \right] \\
\leq~& \left( 1-C_{11}C_0 (1-e^{-1})\right)^{C_7 \frac{n-n'}{N}} \mathbf{P}\left( J_{n',n} \geq C_7 \frac{n-n'}{N} ~\bigg|~ A_3(2^n)\right),
\end{align*}
and so, returning to \eqref{eq: come_here!}, we conclude that
\[
\mathbf{P}\left( \# I_{n',n} \leq C_9 \frac{n-n'}{N} ~\bigg|~ A_3(2^n) \right) \leq e^{-C_7(n-n')} + e^{C_9 \frac{n-n'}{N}} (1-C_{11}C_0(1-e^{-1}))^{C_7\frac{n-n'}{N}}.
\]
By the inequality $\log(1-x) \leq -x$, we get the upper bound
\[
e^{-C_7(n-n')} + \exp\left( \frac{n-n'}{N} \left[ C_9 - C_{11}C_7C_0(1-e^{-1})\right]\right).
\]
%Noting that for all $t \in [0,1]$, one has $1-e^{-t} \geq t/e$, this becomes
%\[
%\mathbf{P}\left( \# I_{n',n} \leq C_3 \frac{n-n'}{N} \right) \leq e^{-C_1(n-n')} +\exp\left( \frac{n-n'}{N} t (C_3 - C_4C_1C_0/e)\right).
%\]
We therefore choose $C_9 = C_{12}C_0$, where $C_{12} = \min\left\{ 1, C_{11}C_7(1-e^{-1})/2\right\}$ to obtain the bound
\[
\mathbf{P}\left( \#I_{n',n} \leq C_{12} C_0 \frac{n-n'}{N} ~\bigg|~ A_3(2^n) \right) \leq e^{-C_7(n-n')} + \exp\left( -C_{12}C_0 \frac{n-n'}{N}\right).
\]
This implies for some universal $C_{13}>0$,
\[
\mathbf{P}\left( \#I_{n',n} \leq C_{13}C_0 \frac{n-n'}{N} ~\bigg|~ A_3(2^n) \right) \leq \exp\left( - C_{13}C_0 \frac{n-n'}{N}\right).
\]

\end{proof}

\end{document}